\documentclass[a4paper,12pt]{amsart}

\usepackage[utf8]{inputenc}
\usepackage[T1]{fontenc}
\usepackage[pagebackref]{hyperref}
\usepackage{amsfonts}
\usepackage{bbm}
\usepackage{mathtools}
\usepackage{enumerate}
\usepackage{ulem}
\usepackage[left=2cm,right=2cm]{geometry}
\allowdisplaybreaks
\newtheorem{theorem}{Theorem}[section]
\newtheorem{lemma}[theorem]{Lemma}
\newtheorem{prop}[theorem]{Proposition} % added
\theoremstyle{definition}

\theoremstyle{remark}
\newtheorem{remark}[theorem]{Remark}

\numberwithin{equation}{section}

\newcommand{\ind}[1]{\mathbf{1}_{\{#1\}}}
\newcommand{\E}{\mathbb{E}} % expected value
\newcommand{\p}{\mathbb{P}} % probability
\newcommand{\R}{\mathbb{R}} % real numbers

\DeclareMathOperator{\Med}{Med}	% median

\title{On Orlicz spaces satisfying the Hoffmann-Jørgensen inequality}

\author{Rados{\l}aw Adamczak}
\address{Institute of Mathematics, University of Warsaw, ul. Banacha 2, 02-097 Warszawa, Poland}
\email{r.adamczak@mimuw.edu.pl}
\author{Dominik Kutek}
\address{Institute of Mathematics, University of Warsaw, ul. Banacha 2, 02-097 Warszawa, Poland}
\email{d.kutek@mimuw.edu.pl}

\date{}

\begin{document}
\maketitle

\begin{abstract}
Building on Talagrand's proof of the Hoffmann-Jørgensen inequality for $L_p$ spaces and its version for the exponential Orlicz spaces we provide a full characterization of Orlicz functions $\Psi$ for which an analogous inequality holds in the Orlicz space $L_\Psi(F)$, where $F$ is an arbitrary Banach space.

As an application we present a characterization of Talagrand-type concentration inequality for suprema of empirical processes with envelope in $L_\Psi$ (equivalently for sums of independent $F$-valued random variables in $L_\Psi(F)$). This result generalizes in particular an inequality by the first-named author concerning exponentially integrable summands and a recent inequality due to Chamakh--Gobet--Liu on summands with $\beta$-heavy tails. Another corollary concerns concentration for convex functions of independent, unbounded random variables, generalizing recent results due to Klochkov--Zhivotovskiy and Sambale.

We also obtain a corollary concerning boundedness in $L_\Psi(F)$ of partial sums of a series of independent random variables, generalizing the original result by Hoffmann-Jørgensen.
\end{abstract}

\section{Introduction}

In the 1970s Jørgen Hoffmann-Jørgensen \cite{MR0356155} established moment inequalities, which over the years have become an important tool in functional analysis and probability theory, with applications to integrability properties of sums of independent random variables as well as stochastic processes with independent increments, convergence of series of independent variables, probabilistic characterizations of certain classes of Banach spaces, the theory of empirical processes and concentration of measure \cite{MR1102015,MR1385671,MR1857312,MR2123200,MR2434306,MR2424985,chamakh:hal-03175697}.
The inequalities are of the form
\begin{align}\label{eq:HJ-mean}
\Big\|\sum_{i=1}^N X_i \Big\|_p \le D_p \Big( \Big\|\sum_{i=1}^N X_i \Big\|_1 + \Big\|\max_{i\le N} \|X_i\|\Big\|_p\Big),
\end{align}
where $X_1,\ldots, X_N$ are independent random variables with values in a Banach space $(F, \|\cdot\|)$ and $p\ge 1$.

In fact, Hoffmann-Jørgensen proved a bound of the form
\begin{align}\label{eq:HJ-quantile}
\Biggr\|\max_{n\le N} \Big\|\sum_{i=1}^n X_i\Big\| \Biggr\|_p \le \widetilde{D}_p \Big(t_0 + \Big\|\max_{i\le N} \|X_i\|\Big\|_p\Big),
\end{align}
 where $t_0$ is an appropriate quantile of  $\max_{n\le N} \|X_1+\ldots+X_n\|$, more precisely
\begin{displaymath}
 t_0 = \inf\{t>0\colon \p(\max_{n\le N} \|X_1+\ldots+X_n\| > t) \le (2\cdot 4^p)^{-1}\}
\end{displaymath}
(see \cite[Proposition 6.8]{MR1102015} for this formulation).

Given the inequality \eqref{eq:HJ-quantile} for $p=1$, it is not difficult to show that the estimates \eqref{eq:HJ-mean} and \eqref{eq:HJ-quantile} are formally equivalent up to the value of the constants $D_p, \widetilde{D}_p$ (see the proof of Theorem \ref{thm:boundedness} below). For this reason we will refer to both of them as to Hoffmann-Jørgensen inequalities.

The constants $D_p$ obtained in \cite{MR0356155} depend exponentially on $p$. It was proven by Michel Talagrand \cite{MR1048946} that as $p \to \infty$ the optimal constants grow like $\frac{p}{\log p}$ (see also the work \cite{MR0770640} by Johnson, Schechtman and Zinn, where this statement is proved for the real-valued case of related Rosenthal's inequality). Talagrand's argument relied on a novel isoperimetric approach, which opened the way to many new concentration results in product spaces.
In particular, Talagrand obtained extensions of the inequality \eqref{eq:HJ-mean} from $L_p$ spaces to exponential Orlicz spaces, which we will now briefly review. Recall that a function $\Psi \colon [0,\infty) \to [0,\infty)$ is called an Orlicz function\footnote{We remark that various naming conventions exist, sometimes the functions we consider are called Young functions or $N$-functions, moreover all three names appear in the literature with subtle differences in definitions (sometimes one imposes additional conditions at zero, allows the value $\infty$, relaxes the strict monotonicity or even convexity conditions). For simplicity we will use the name \emph{Orlicz function} in analogy with \emph{Orlicz space}.} if it is convex, strictly increasing and $\Psi(0) = 0$. For such a function and an $F$-valued random variable $X$ we define
\begin{align}\label{eq:Orlicz-norm}
\|X\|_{\Psi} = \inf\{a> 0\colon \E \Psi(\|X\|/a) \le 1\}.
\end{align}

One then shows that for a probability space $(\Omega,\mathcal{F},\p)$, $\|\cdot\|_\Psi$ defines a norm on the space $L_\Psi(F)$ of all random variables $X\colon \Omega \to  F$, such that $\|X\|_{\Psi} < \infty$. Moreover, $(L_\Psi(F),\|\cdot\|_\Psi)$ is a Banach space. In \cite{MR1048946} Talagrand considered the function $\Psi_1(x) = e^x - 1$ and its generalizations $\Psi_\alpha(x) = e^{x^\alpha} - 1$ for $\alpha > 0$ (when $\alpha \in (0,1)$, $\Psi_\alpha$ is not an Orlicz function, as it is not convex, and $\|\cdot\|_{\Psi_\alpha}$ defined by \eqref{eq:Orlicz-norm} is only a quasi-norm, however $\Psi_\alpha$ can be modified for small values of $x$ to become an Orlicz function and such a modification leads to a norm equivalent to $\|\cdot\|_{\Psi_\alpha}$; for this reason in the literature one often refers to $\|\cdot\|_{\Psi_\alpha}$ as a norm). We will call the norm $\|\cdot\|_{\Psi}$ the Orlicz norm associated with $\Psi$, as it is common in the probabilistic literature (see, e.g., \cite{MR1857312,MR3101846,MR3707425}). We note, however, that in functional analysis it is often referred to as the Luxemburg norm, with the name Orlicz being used for another norm, defined by duality and equivalent to $\|\cdot\|_\Psi$ up to a multiplicative constant 2.

Talagrand proved that if $\Psi = \Psi_\alpha$ and $\alpha \in (0,1]$, then there exists $D = D(\alpha)$, depending only on $\alpha$, such that for all independent $F$-valued random variables $X_1,\ldots,X_N$,
\begin{align}\label{eq:HJTal}
\Big\|\sum_{i=1}^N X_i \Big\|_{\Psi} \le D\Big( \Big\|\sum_{i=1}^N X_i \Big\|_1 + \Big\|\max_{i\le N} \|X_i\|\Big\|_{\Psi}\Big).
\end{align}

He also provided a characterization of Orlicz functions $\Psi$ of the form $\Psi(x) = x\xi(x)$, where $\xi$ is increasing for large $x$, which satisfy \eqref{eq:HJTal}. It turns out that \eqref{eq:HJTal} holds iff there exists $L$ such that for all $u$ large enough $\xi(e^u) \le L\xi(u)$. Formally, Talagrand's result covers random variables with mean zero, but it is easy to see that this assumption can be dropped at a cost of increasing $D$ by a constant depending only on $\Psi$ (see the proof of Theorem \ref{thm:characterization} below).

Recently, the interest in inequalities of the form \eqref{eq:HJTal} was revived in the context of empirical processes and statistical learning without exponential moments. Chamakh, Gobet and Liu in \cite{chamakh:hal-03175697} derived \eqref{eq:HJTal} for $\Psi(x) = \exp(\log^\beta(1+x)) - 1$, where $\beta \ge 1$. This corresponds to the integrability of log-normal variables or more generally exponentials of Weibull variables.

In view of these results it is natural to ask whether there exists a simple characterization of all functions $\Psi$, which satisfy the inequality \eqref{eq:HJTal} for all Banach spaces $F$ and sequences of independent $F$-valued random variables.

The main result of this article gives an affirmative answer to this question, the full characterization is provided in Theorem \ref{thm:characterization} below. It allows us to obtain refined concentration inequalities for unbounded random variables. In particular, we generalize inequalities from \cite{MR2424985,chamakh:hal-03175697,MR2591905} and obtain estimates for suprema of empirical processes (equivalently, sums of independent $F$-valued random variables) in $L_\Psi$,  involving the weak variance as the coefficient in the subgaussian part of the estimate, as in the by now classical inequality due to Talagrand in the bounded case \cite{MR1419006}. As another corollary we derive concentration inequalities for convex Lipschitz functions of random vectors distributed according to product measures with unbounded support, extending results from \cite{MR4073683,10.1007/978-3-031-26979-0_7}. We also provide applications concerning Orlicz integrability of series of independent random variables in Banach spaces, generalizing the results by Hoffmann-Jørgensen \cite{MR0356155}.

The organization of the article is as follows. After introducing the notation and preliminary facts (Section 2) we formulate our main results concerning characterization of the Hoffmann-Jørgensen inequality (Section \ref{sec:main-results}) followed by the aforementioned applications to concentration of measure and probability in Banach spaces (Section \ref{sec:applications}). The proofs of all results are presented in Section \ref{sec:proofs}.

\section{Notation and preliminaries \label{sec:notation}}

In what follows by $C,c$ we will denote constans which are either universal or depend only on the function $\Psi$ or some parameters related to it, but do not depend on the Banach space $F$, the law of the variables $X_i$ or their number $N$. The set of parameters on which the constant may depend will be explained in the text. Sometimes, to stress the dependence we will write, e.g., $C_p, c_\Psi$. The value of such constants may change between occurrences.

To simplify the notation, with an Orlicz function $\Psi\colon[0,\infty) \to [0,\infty)$ we will always associate a function $\psi\colon [0,\infty) \to [0,\infty)$, defined by the equality
\begin{displaymath}
  \Psi(x) = e^{\psi(x)} - 1, \; x\ge 0.
\end{displaymath}
It follows from the definition of the Orlicz function that $\psi$ is increasing, continuous, $\psi(0) = 0$, $\lim_{x\to\infty}\psi(x) = \infty$. Moreover, there exists $c = c_{\Psi} > 0$ such that for large $x$, $\psi(x) \ge \ln(1+cx)$.

For a Banach space $(F,\|\cdot\|)$, by an $F$-valued random variable on a probability space $(\Omega,\mathcal{F},\p)$ we mean a function $X\colon \Omega \to F$, which is $\mathcal{B}(F)/\mathcal{F}$ measurable, where $\mathcal{B}(F)$ is the Borel $\sigma$-field generated by the norm $\|\cdot\|$. An $F$-valued random variable is called Radon if it takes values in a separable subspace of $F$. We refer to \cite{MR1102015} for basic properties of Radon random variables. To avoid technical measurability issues, in what follows we will restrict our attention to Radon variables. In particular, we will use the phrases \emph{Radon random variable} and \emph{random variable} interchangeably.

For an Orlicz function $\Psi$, by $L_\Psi(F)$ we will denote the space of all $F$-valued Radon random variables, such that $\|X\|_{\Psi} < \infty$ (where $\|X\|_{\Psi}$ is defined in \eqref{eq:Orlicz-norm}). It is well known that $(L_{\Psi}, \|\cdot\|_{\Psi})$ is a Banach space. In particular, for $\Psi(x) = x^p$ we have $\|X\|_{\Psi} = \|X\|_p = (\E \|X\|^p)^{1/p}$ and $L_{\Psi}(F) = L_p(F)$. In what follows we will most of the time suppress $F$ from the notation and write simply $L_p$ or $L_\Psi$ instead of $L_p(F)$, $L_{\Psi}(F)$. The meaning of this short-hand notation will be clear from the context. We refer to the monographs \cite{MR0126722,MR1113700} for an overview of the theory of Orlicz spaces.

\section{Main results \label{sec:main-results}}

\subsection{Orlicz functions satisfying the Hoffmann-Jørgensen inequality.}\label{Sec:characterization}

Let $\Psi$ be an Orlicz function. Recall that the function $\psi\colon [0,\infty)\to [0,\infty)$ is defined via $\Psi(x) = e^{\psi(x)} - 1$. We will say that $\Psi$ satisfies the condition \eqref{eq:nsc} if there exists a constant $K>0$ such that for every $s,u \ge K$,
\begin{align}\label{eq:nsc}\tag{HJ}
\psi(su) \le K(s\ln (1+s) + s\psi(u)).
\end{align}

\begin{remark}\label{re:nsc1} One can easily see that the condition \eqref{eq:nsc} is equivalent to the existence of $K' > 0$ such that for all $s,u \ge K'$,
\begin{align}\label{eq:nsc1}\tag{HJ'}
\Psi(su) \le (s\Psi(u))^{K's}.
\end{align}
\end{remark}

It turns out that the condition \eqref{eq:nsc} provides a full characterization of the Hoffmann-Jørgensen inequality. More precisely, the following theorem holds.

\begin{theorem}\label{thm:characterization}
If the Hoffmann-Jørgensen inequality \eqref{eq:HJTal} holds for all sequences of i.i.d., symmetric, real valued, bounded random variables, then $\Psi$ satisfies \eqref{eq:nsc}.
Conversely, if \eqref{eq:nsc} holds, then there exists $D<\infty$, depending only on $\Psi$, such that the inequality \eqref{eq:HJTal} holds for all sequences of independent Banach space valued Radon random variables.
\end{theorem}

\subsection{Examples\label{sec:examples}}
Let us now provide several examples of Orlicz functions satisfying the Hoffmann-Jørgensen inequality.

Observe first that if $\Psi(x) = e^{\psi(x)} - 1$ and there exists $x_0$ such that $\psi$ is concave on $[x_0,\infty)$, then \eqref{eq:nsc} is clearly satisfied, since by concavity for $s, u \ge \max(1,2x_0)$, we have
\begin{displaymath}
  \psi(su) \le \psi(x_0) + \frac{su-x_0}{u-x_0}(\psi(u) - \psi(x_0)) \le \psi(x_0) + 2s\psi(u) \le 3s\psi(u).
\end{displaymath}

In particular, the functions $\Psi_\alpha(x) = e^{x^\alpha} - 1$ for $\alpha \le 1$, considered by Talagrand in \cite{MR1048946}, satisfy \eqref{eq:nsc}. As mentioned in the introduction, for $\alpha < 1$ they are not convex, but can be modified to Orlicz functions $\widetilde{\Psi}_\alpha$ in such a way that the norm $\|\cdot\|_{\widetilde{\Psi}_\alpha}$ is equivalent to the quasi-norm $\|\cdot\|_{\Psi_\alpha}$.

The above observation covers also the functions
\begin{align}\label{eq:beta-Orlicz}
\Psi_{\beta}^{HT}(x) = e^{\ln^\beta(x+1)} - 1, \; \beta \ge 1,
\end{align}
considered recently by Chamakh, Gobet and Liu \cite{chamakh:hal-03175697}. This class of functions is of interest, since the finiteness of $\|X\|_{\Psi_\beta^{HT}}$ implies the existence of all moments, but does not imply the existence of exponential moments of any order, so the tails of $X$ may be considered moderately heavy. In particular, the case $\beta = 2$ captures the log-normal distribution, and typical examples one may have in mind in general are exponentials of Weibull variables. We refer to \cite{chamakh:hal-03175697} for more examples and applications to statistical learning based on heavy-tailed observables.

Let us observe that the above examples can be also extended to functions $\Psi(x) = e^{\psi(x)} - 1$, where $\psi(x) = x^\alpha f(x)$, $\alpha \in (0,1)$ or $ \psi(x) = \log^\beta(1+x)f(x)$, $\beta \ge 1$, and $f(x)$ is any finite product of iterated logarithms, as such functions satisfy $\psi(su) \le s\psi(u)$ for $s,u$ large enough and thus the condition \eqref{eq:nsc}.

Setting $\psi(x) = \log(1+x^p)$ we get $\Psi(x) = x^p$, and in particular we recover the original inequalities \eqref{eq:HJ-mean} due to Hoffmann-Jørgensen. We note that an argument based only on the condition $\eqref{eq:nsc}$ does not allow to recover the optimal growth of constants $D_p$, since the condition \eqref{eq:nsc} does not capture the full regularity of the function $\psi$ in this case, i.e., the fact that $\psi(su)$ for $u$ and $s$ large may be bounded by a much smaller quantity than the right hand side of \eqref{eq:nsc}.

In \cite{MR1464686} Ziegler, using the original proof of Hoffmann-Jørgensen, obtained \eqref{eq:HJTal} for $\Psi$ satisfying the $\Delta_2(\infty)$--condition
\begin{displaymath}
  \Psi(2u) \le C\Psi(u)
\end{displaymath}
for some $C < \infty$ and all $u$ large enough (we remark that this is a very common assumption in the theory of Orlicz spaces, known to be equivalent to many functional-analytic and probabilistic properties of the space $L_\Psi$). Clearly, if $\Psi$ satisfies this property, then there exist constants $K,p$ such that $\Psi(x) \le K x^p$ for $x > K$. It turns out that this growth condition on $\Psi$, weaker than the $\Delta_2(\infty)$-condition, is actually sufficient for \eqref{eq:nsc} and consequently for \eqref{eq:HJTal} to hold. Indeed, by convexity of $\Psi$, there exists $c > 0$ such that $\Psi(x) \ge cx$ for $x$ large enough, and so for some constant $C$ and $u,s$ large enough,
\begin{displaymath}
\frac{1}{2}\psi(us) \le \ln (\Psi(us)) \le p\ln(us) + \ln(K) \le C s\ln (u),
\end{displaymath}
which clearly implies \eqref{eq:nsc}. We thus obtain a generalization of Ziegler's result (note that there are Orlicz functions of polynomial growth which do not satisfy the $\Delta_2$ condition). We remark that Ziegler worked in fact in a more general setting of (not necessarily measurable) suprema of empirical processes and outer expectations, while we restrict our attention to Radon random variables. One should be able to establish our results also in Ziegler's setting, this would however require additional technicalities related mostly to the lack of measurability and the need to work with outer integrals.

 In view of the last class of examples, one may ask if there is some function $\Phi$ of superpolynomial growth such that the Hoffman-Jørgensen inequality holds for all Orlicz functions $\Psi$ satisfying $\Psi(x) \le \Phi(x)$ for $x$ large enough. This turns out not to be the case as explained in the following proposition, whose proof is deferred to Section \ref{sec:proofs-main-results}.

\begin{prop}\label{prop:counterexample}
Assume that $\Phi$ is an Orlicz function such that $\limsup_{x \to +\infty} \frac{\Phi(x)}{x^p} = +\infty$ for all $p > 0$. Then, there exists an Orlicz function $\Psi$, such that $\Psi(x) \le \Phi(x)$ for all $x$ large enough and $\Psi$ does not satisfy the condition \eqref{eq:nsc}.
\end{prop}

\section{Applications \label{sec:applications}}
\subsection{Talagrand-type inequalities involving weak variance\label{sec:weak-variance}}

Inequalities of Hoffmann-Jørgensen type provide strong integrability for sums of Banach space valued random variables, however in general do not capture fully their concentration properties. Consider for instance a situation, in which $X_1,X_2,\ldots,$ are i.i.d. and $S_N = X_1+\ldots+X_N$, normalized by $\sqrt{N}$ converges in distribution to an $F$-valued Gaussian random variable $G$ (we will not discuss here conditions under which such convergence holds and refer the reader to the classical monographs \cite{MR0576407,MR1102015}). Thanks to the Gaussian concentration inequality (see, e.g., \cite{MR1849347,MR1102015}), the limiting variable satisfies
\begin{displaymath}
  \p\Big(\Big|\|G\|- \E \|G\|\Big|\ge t \Big) \le 2\exp\Big(-\frac{t^2}{2\sigma^2}\Big),
\end{displaymath}
where $\sigma^2 = \sup_{f \in F^\ast, \|f\|\le 1} \E f(G)^2$ is the \emph{weak variance} of $G$ ($F^\ast$ is the dual of $F$, for simplicity we denote the dual norm in $F^\ast$ also by $\|\cdot\|$). In analogy with one-dimensional Bernstein and Bennett's inequalities (see, e.g., chapters 2.7, 2.8 in the monograph \cite{MR3185193}) one expects that a similar inequality should hold for $S_n$, i.e. that one can obtain an inequality for the tail of $|\|S_N\| - \E\|S_N\||$ which, at least for \emph{moderate} values of $t$, would exhibit a subgaussian behaviour, with the subgaussian coefficient $\sqrt{N} \sup_{\|f\| \le 1} \E f(X_1)^2$. In the case of bounded variables this has been achieved by Talagrand \cite{MR1419006} as another example of application of his strong isoperimetric approach. More precisely, Talagrand established the following result \cite[Theorem 1.4]{MR1419006}.

\begin{theorem}\label{thm:Talagrand-inequality}
Let $X_1,\ldots,X_N$ be independent random variables with values in a measurable space $(\mathcal{S},\mathcal{G})$ and let $\mathcal{A}$ be a countable class of measurable real valued bounded functions on $\mathcal{S}$. Consider the random variable
\begin{displaymath}
  S = \sup_{f\in \mathcal{A}} \sum_{i=1}^N f(X_i)
\end{displaymath}
and define
\begin{displaymath}
 U = \sup_{f \in \mathcal A}\|f\|_{\infty}
\end{displaymath}
and
\begin{displaymath}
  \Sigma^2 = \E \sup_{f \in \mathcal A} \sum_{i=1}^N f(X_i)^2.
\end{displaymath}
Then, for every $t > 0$,
\begin{align}\label{eq:Talagrand-empirical}
  \p(|S - \E S| \ge t) \le C\exp\Big(-\frac{t}{CU}\log\Big(1+\frac{tU}{\Sigma^2}\Big)\Big),
\end{align}
where $C$ is an absolute constant.
\end{theorem}

Taking $\mathcal{S} = F$ with $\mathcal{A}$  being the unit ball in the dual of the separable subspace of $F$, containing the range of $X_i$'s, we obtain an inequality in our setting of Radon variables. For $\mathcal{A}$ consisting of only one function, Talagrand's result reduces up to constants to Bennett's inequality and recovers the Gaussian behaviour of the tail of a sum of independent bounded random variables for small values of $t$ (note that $tU^{-1}\log(1+ tU/\Sigma^2) \simeq t^2/\Sigma^2$ for $tU/\Sigma^2$ small) and the Poisson behaviour for larger $t$.

It may initially seem that the parameter $\Sigma^2$ is not of weak type, since the supremum appears inside the expectation, however under an additional assumption that $\E f(X_i) = 0$ for all $f \in \mathcal{A}$ and $i \le N$, thanks to Talagrand's contraction principle \cite[Theorem 4.12]{MR1102015} and the classical Hoffmann-Jørgensen inequality one may estimate
\begin{align}\label{eq:strong-weak-variance}
  \Sigma^2 \le \sup_{f \in \mathcal{A}} \sum_{i=1}^N \E f(X_i)^2 + 32\sqrt{\E \max_{i\le N}\sup_{f\in\mathcal{A}} |f(X_i)|^2} \E\sup_{f \in \mathcal{A}}\Big|\sum_{i=1}^N f(X_i)\Big| +  8\E \max_{i\le N} \sup_{f\in\mathcal{A}} |f(X_i)|^2,
\end{align}
which allows to replace $\Sigma^2$ by the weak variance. Estimate of this type appeared first in the work by Talagrand for bounded variables and then in full generality were obtained by Gin\'e, Lata{\l}a and Zinn in \cite{MR1857312}. The above version is taken from \cite{MR2123200}. If the class $\mathcal{A}$ is symmetric around zero, then $\sup_{f \in \mathcal{A}}|\sum_{i=1}^N f(X_i)| = S$ and so Talagrand's inequality together with \eqref{eq:strong-weak-variance} allows to obtain deviation inequalities above $(1+\varepsilon) \E S$ or below $(1-\varepsilon)\E S$, expressed only in terms of the weak variance and the $L_\infty$-bound on the class $\mathcal{A}$ (see, e.g., \cite{MR1782276}). For this reason in what follows we will not focus on the weak parameter but we will express our inequalities in terms of $\Sigma^2$. We remark, that since the publication of Talagrand's result, several other approaches to concentration for bounded empirical processes were introduced, most notably ones based on the entropy method, leading to versions of the inequality \eqref{eq:Talagrand-empirical} with explicit and in some situations optimal constants (see, e.g., \cite{MR1399224,MR1782276,MR2135312,MR1890640}). Since our goal is to illustrate applications of Hoffmann-Jørgensen type inequalities \eqref{eq:HJTal} and we allow our constants to depend on the function $\Psi$, we will not discuss them in detail. Let us only mention that the proofs of our results, combined with the knowledge of constants in \eqref{eq:HJTal} and some additional calculations should provide explicit constants in the inequalities we are to discuss. With some more effort one can obtain estimates with the constants in front of the weak variance arbitrarily close to 2, at the cost of the remaining constants (see, e.g. \cite{MR3423302,MR2424985} for examples of such calculations).

Concentration inequalities for suprema of unbounded empirical processes or sums of unbounded $F$-valued variables, counterparting Theorem \ref{thm:Talagrand-inequality}, appeared in \cite{MR2123200,MR2424985,MR3423302,MR2434306,MR2591905,MR3263097,MR4138129,chamakh:hal-03175697}. They have found applications to the laws of iterated logarithm, estimates of bounded empirical processes of ergodic Markov chains, error bounds in statistical learning theory. Their proofs rely on an appropriate truncation allowing for a reduction to Talagrand's inequality, with the unbounded part handled by other means, usually by an appropriate version of the Hoffmann-Jørgensen inequality.

Below we present a general version of an inequality of this type followed by a discussion of optimality of the assumptions.

\begin{theorem}\label{thm:weak-variance}
Consider an Orlicz function $\Psi$ satisfying \eqref{eq:nsc}. Let $X_1,\ldots,X_N$ be independent random variables with values in a measurable space $(\mathcal{S},\mathcal{G})$ and let $\mathcal{A}$ be a countable class of measurable real valued bounded functions on $\mathcal{S}$. Consider the random variable
\begin{displaymath}
  S = \sup_{f\in \mathcal{A}} \sum_{i=1}^N f(X_i).
\end{displaymath}
Then, there exists a constant $c_\Psi > 0$, depending only on $\Psi$, such that both of the following inequalities hold for any $t>0$:
\begin{equation} \label{eq:Bennett}
\p  \Big(\Big|S - \E S\Big| \ge  t\Big) \le 2\exp\Big(- \frac{c_\Psi t}{U} \ln\Big(1 + \frac{tU}{\Sigma^2}\Big)\Big) + \frac{2}{\Psi\Big( \frac{c_\Psi t}{U}\Big) + 1}
\end{equation}
and
\begin{equation}\label{eq:Bernstein}
\p  \Big(\Big|S - \E S\Big|\ge  t\Big)
 \le 2\exp\Big( - \frac{c_\Psi t^2}{\Sigma^2 + tU}\Big) + \frac{2}{\Psi\Big( \frac{c_\Psi t}{U}\Big) + 1},
\end{equation}
where $U= \|\max_{i \le N}\sup_{f\in \mathcal{A}} |f(X_i)|\|_{\Psi}$ and $\Sigma^2 = \E \sup_{f\in \mathcal{A}} \sum_{i=1}^N f^2(X_i)$.
\end{theorem}

\begin{remark} \label{re:Bennett-Bernstein} The inequality \eqref{eq:Bernstein}, up to the value of the constant, follows from \eqref{eq:Bennett} but we state it explicitly, since it counterparts the classical Bernstein inequality for sums of independent random variables, in particular for functions $\Psi$ which grow at most exponentially it is equivalent (again up to the value of the constant $c_\Psi$) to
\begin{align}\label{eq:Bernstein-equivalent}
\p  \Big(\Big| S - \E S\Big| \ge  t\Big)
 \le 2\exp\Big( - c_\Psi\frac{ t^2}{\Sigma^2}\Big) + \frac{2}{\Psi\Big( \frac{c_\Psi t}{U}\Big) + 1}.
\end{align}
In this case the right-hand sides of \eqref{eq:Bernstein} and \eqref{eq:Bennett} are in fact equivalent up to a constant multiplicative factor. On the other hand, if $\Psi$ grows faster than exponentially (which in view of Theorem \ref{thm:characterization} or earlier results by Talagrand is possible under \eqref{eq:nsc}), then this equivalence does not hold and \eqref{eq:Bennett} may give better results than \eqref{eq:Bernstein}.

\end{remark}

\begin{remark}
Up to the value of the constant $c_\Psi$, the inequality \eqref{eq:Bernstein} together with \eqref{eq:strong-weak-variance} allows to recover several results from the literature (in particular, the results from \cite{MR2434306} and \cite{MR2591905} concerning $\Psi(x) = x^p$, from \cite{MR2424985} on $\Psi_\alpha(x) = e^{x^\alpha} - 1$, and from \cite{chamakh:hal-03175697} on $\Psi(x) = \exp(\log^\beta(1+x)) - 1$), which have found applications to strong limit theorems, statistics, inequalities for additive functionals of Markov chains.
\end{remark}

We will now address the necessity of the assumption \eqref{eq:nsc} in Theorem \ref{thm:weak-variance}.

\begin{prop}\label{prop:weak-variance-necessity}  Assume that $\Psi$ is an Orlicz function such that for all sequences of i.i.d real valued, mean zero random variables $X_1,\ldots,X_N$ with $\|X_1\|_\infty \le 1$ and all $t > 0$,
\begin{multline} \label{eq:Bennett-real-valued}
\p\Big(\sum_{i=1}^N X_i \ge  t\Big) \\
\le C_\Psi\exp\Big(- c_\Psi \frac{t}{\| \max_{i\le N} |X_i|\|_\Psi} \ln\Big(1 + \frac{t\| \max_{i\le N} |X_i|\|_\Psi}{N\E X_1^2}\Big)\Big) + \frac{C_\Psi}{\Psi\Big( c_\Psi \frac{t}{\| \max_{i\le N} |X_i|\|_\Psi}\Big) + 1},
\end{multline}
where $C_\Psi, c_\Psi$ are positive constants depending only on $\Psi$. Then $\Psi$ satisfies the condition \eqref{eq:nsc}.

Moreover, the condition \eqref{eq:nsc} also holds if there exist positive constants $c_\Psi, C_\Psi < \infty$ such that $\Psi(x) \le C_\Psi e^{C_\Psi x}$ for all $x \ge 0$ and  for all sequences $(X_n)_{n\le N}$ as above and all $t > 0$,
\begin{equation}\label{eq:Bernstei-real-valued}
\p\Big(\sum_{i=1}^N X_i \ge  t\Big)  \le C_{\Psi}\exp\Big( - \frac{c_\Psi t^2}{N\E X_1^2 + t\| \max_{i\le N} |X_i|\|_\Psi}\Big) + \frac{C_\Psi}{\Psi\Big( c_\Psi \frac{t}{\| \max_{i\le N} |X_i|\|_\Psi}\Big) + 1}.
\end{equation}

\end{prop}

\begin{remark}
In the second part of Proposition \ref{prop:weak-variance-necessity} one needs to assume a growth condition on $\Psi$ since clearly for any function of superexponential growth we have $\|\cdot\|_{\Psi} \ge c\|\cdot\|_{\Psi_1}$ (recall that $\Psi_1(x) = e^x - 1$), moreover the second term on the right-hand side of \eqref{eq:Bernstein} is dominated by the first one and so \eqref{eq:Bernstein} follows from its version for $\Psi_1$.
\end{remark}

\subsection{Inequalities for convex functions\label{sec:convex-functions}}

 Another celebrated result by Talagrand is the concentration inequality for convex functions of independent bounded random variables \cite{MR1387624}, which asserts that if $f\colon \R^N\to \R$ is a convex 1-Lipschitz function and $X_1,\ldots,X_N$ are independent random variables with values in $[-1,1]$, then for all $t\ge 0$,

\begin{align}\label{eq:Talagrand-convex}
  \p(|f(X_1,\ldots,X_N) - \Med f(X_1,\ldots,X_N)|\ge t) \le 4e^{-t^2/16}.
\end{align}

One of the important observations related to this result and responsible for many of its applications is that contrary to dimension free concentration for all $1$-Lipschitz functions, which requires strong geometric assumptions on the law of the variables $X_i$ (see, e.g., \cite{MR2749436}), the only assumption required for \eqref{eq:Talagrand-convex} is the uniform boundedness of $X_i$. Following the result by Talagrand, considerable attention has been devoted to extending the inequality \eqref{eq:Talagrand-convex} to unbounded variables. One could think that the uniform boundedness can be replaced with a uniform bound on $\|X_i\|_{\Psi_2}$, however it turns out that dimension free concentration again requires geometric conditions on the law of $X_i$, the difference with the result for Lipschitz functions being that the regularity is required only outside of a compact subset of $\R$ (see the series of papers \cite{MR2749436,MR3706606,MR3825894} by Gozlan et al., in which this question is extensively discussed). In another direction Klochkov and Zhivotovskiy \cite{MR4073683} (see also \cite{MR2391154}) and Huang and Tikhomirov \cite{MR4583676} investigated inequalities which can be obtained just under a bound on $\|X_i\|_{\Psi_2}$. Such estimates contain in general some logarithmic factors or are expressed in terms of $\|\max_{i\le n} |X_i|\|_{\Psi_2}$. Similar (but necessarily weaker) inequalities were also obtained by Huang and Tikhomirov \cite{MR4583676} and by Sambale \cite{10.1007/978-3-031-26979-0_7} for variables with bounded $\|\cdot\|_{\Psi_\alpha}$ norm for $\alpha \in (0,2)$.

The proposition below generalizes results due to Klochkov--Zhivotovskiy \cite{MR4073683} and Sambale \cite{10.1007/978-3-031-26979-0_7} to more general Orlicz norms, $\|\cdot\|_{\Phi}$.

\begin{prop}\label{prop:convex}
Let $\Psi$ be an Orlicz function satisfying \eqref{eq:nsc} and let $\Phi$ be the Orlicz function defined by $\Phi(x) = \Psi(x^2)$, $x \ge 0$. If $X=(X_1,...,X_N)$ is an $\R^N$ valued random variable with independent coordinates $X_1,...,X_N$, such that $\| \max_{i \le N} |X_i| \|_{\Phi} < \infty$, then
\begin{align}\label{eq:convex-unbounded}
\p  \Big( |f(X) - \E f(X)| > t \Big)
\le 2\exp\Big(-c_\Phi\frac{t^2}{(\E \max_{i \le N} |X_i| )^2 }\Big) +  \frac{2}{ \Phi\Big( c_\Phi \frac{t}{\| \max_{i \le N} |X_i| \|_{\Phi}}\Big) + 1}
\end{align}
for any convex and 1-Lipschitz function $f:\mathbb R^N \to \mathbb R$ and $t>0$, where $c_\Phi \in (0,\infty)$ is a constant depending only on $\Phi$.
\end{prop}
\begin{remark}
It is easy to show that in the above theorem one can replace the mean by a median, as under \eqref{eq:convex-unbounded} the median and the mean are apart by at most $C_\Psi \|\max_{i\le N} |X_i|\|_{\Phi}$, so it is enough to adjust constants.
\end{remark}

\begin{remark}
We note that unlike in the case of inequalities presented in Section \ref{sec:weak-variance} we do not know if under an additional assumption that $\Psi(x) \le Ce^{Cx}$, the condition \eqref{eq:nsc} is necessary for the inequality \eqref{eq:convex-unbounded} to hold.
\end{remark}

\subsection{Orlicz boundedness of partial sums of random series with independent summands\label{sec:boundedness}}

Let us now pass to an application related to integrability of partial sums of series of independent random variables in Banach spaces, which was the original motivation of Hoffmann-Jørgensen. Recall that a sequence $(Y_n)_{n\ge 1}$ of random variables with values in $(F,\|\cdot\|)$ is said to be stochastically bounded, if for every $\varepsilon > 0$ there exists $K > 0$ such that for all $n \ge 1$, $\p(\|Y_n\| \ge K) < \varepsilon$. In \cite{MR0356155} Hoffmann-Jørgensen proved the following result (see Theorem 3.1 and Corollary 3.2 therein).

\begin{theorem}\label{thm:HJ}
Let $X_n$, $n\ge 1$ be independent random variables with values in a Banach space $(F,\|\cdot\|)$ and let $S_n = X_1+\ldots+X_n$. If the sequence $(S_n)_{n\ge 1}$ is stochastically bounded, then the following conditions are equivalent:
\begin{itemize}
\item[(i)] $(S_n)_n$ is bounded in $L_p$,

\item[(ii)] $\sup_n \|X_n\| \in L_p$,

\item[(iii)] $\sup_n \|S_n\| \in L_p$.
\end{itemize}
\end{theorem}

Given the Hoffmann-Jørgensen inequality \eqref{eq:HJ-quantile}, the proof of the above theorem is a nowadays standard application of symmetrization techniques and L\'{e}vy's maximal inequality. A straightforward adaptation of this argument gives the following theorem. For completeness we will provide the proof in Section \ref{sec:proofs-boundedness}.

\begin{theorem}\label{thm:boundedness} Let $X_n$, $n\ge 1$ be independent random variables with values in a Banach space $(F,\|\cdot\|)$ and let $S_n = X_1+\ldots+X_n$. Let $\Psi$ be an Orlicz function satisfying the condition \eqref{eq:nsc}. If the sequence $(S_n)_{n\ge 1}$ is stochastically bounded, then the following conditions are equivalent:

\begin{itemize}
\item[(i)] $(S_n)_n$ is bounded in $L_\Psi$,

\item[(ii)] $\sup_n \|X_n\| \in L_\Psi$,

\item[(iii)] $\sup_n \|S_n\| \in L_\Psi$.
\end{itemize}
\end{theorem}

It turns out that the condition \eqref{eq:nsc} is necessary for the above theorem to hold, as stated in the following proposition, the proof of which is also deferred to Section \ref{sec:proofs-boundedness}.

\begin{prop}\label{prop:boundedness-necessity} Assume that $\Psi$ does not satisfy the condition \eqref{eq:nsc}. Then, there exists a sequence of independent, bounded, symmetric real random variables $(X_n)_{n\ge 1}$, such that:
\begin{itemize}
  \item[(i)] $(S_n)_n$ converges almost surely, in particular is stochastically bounded,
  \item[(ii)] $\sup_n |X_n| \in L_\Psi$,
  \item[(iii)] $(S_n)_n$ is not bounded in $L_\Psi$.
\end{itemize}
\end{prop}

\section{Proofs\label{sec:proofs}}

\subsection{Main technical tools and auxiliary lemmas}

In this section we will present some auxiliary facts to be used in the proof of Theorem \ref{thm:characterization}. Let us start with the following lemma \cite[Lemma 8]{MR1048946}.

\begin{lemma}
Given a sequence $(X_j)_{j \le N}$ of independent $F$-valued random variables and $t > 0$, such that $\p  (\max_{j \le N}\|X_j\| \ge t) \le \frac{1}{2}$, it holds that
\begin{displaymath}
\sum_{j \le N}\p  (\|X_j\| \ge t) \le 2\p  (\max_{j \le N}\|X_j\| \ge t).
\end{displaymath}
\end{lemma}

The next lemma is a standard consequence of the Chebyshev inequality, providing tail estimates for random variables with a finite $\|\cdot\|_{\Psi}$ norm. We skip the easy proof.

\begin{lemma}\label{le:Orlicz-Chebyshev}
Let $Y$ be a nonnegative random variable and $\Psi = e^{\psi}-1$ be an Orlicz function.  Then for every $t > 0$,
\begin{displaymath}
  \p(Y \ge t) \le 2\exp(-\psi(t/\|Y\|_{\Psi})).
\end{displaymath}
\end{lemma}

\medskip

Another result which we will need is a version of \eqref{eq:Talagrand-convex} with better constants and providing concentration around mean, rather than the median. We remark that we could alternatively use just \eqref{eq:Talagrand-convex}, since subgaussian concentration around mean and around median are equivalent, but the version we provide here can be applied in a more straightforward way. It can be easily obtained, e.g.,  from \cite[Corollary 1]{MR1756011} (see also the comments before Corollary 3 therein). We note that \cite{MR1756011} deals in fact with a much more general dependent situation. A version of the theorem below for convex functions was obtained earlier in \cite[Corollary 1.3]{MR1399224}.

\begin{theorem}\label{thm:convex-Ledoux} If $X_1,\ldots,X_n$ are independent random variables with values in $[-1,1]$ and $f\colon [-1,1]^n \to \R$ is a 1-Lipschitz convex or concave function, then for any $t > 0$,
\begin{displaymath}
  \p(f(X_1,\ldots,X_n) \ge \E f(X_1,\ldots,X_n) + t) \le e^{-t^2/8}.
\end{displaymath}
\end{theorem}

The next lemma is a simple estimate on $\psi^{-1}$ for $\psi$ which satisfies the condition \eqref{eq:nsc}.

\begin{lemma} \label{le:phi-inverse} Let $\Psi=e^\psi-1$ be strictly increasing and satisfy the condition \eqref{eq:nsc}. Then, $\widehat{K}\ln(1+x)\psi^{-1}(xy) \ge x\psi^{-1}(y)$ for some $\widehat{K}>0$ and every $x,y \ge \widehat{K}$.
\end{lemma}

\begin{proof} Since $\psi$ is strictly increasing so is $\psi^{-1}$, and hence $su \le \psi^{-1}\Big( K \Big(s\ln(1+s) + s\psi(u)\Big)\Big)$. Let $t = Ks$ and $y = \psi(u)$ to get $t \psi^{-1}(y) \le K \psi^{-1}\Big( t\ln\Big(1+\frac{t}{K}\Big) + ty\Big) \le K\psi^{-1}(t\ln(1+t)y)$, where in the last inequality we assumed without loss of generality that $K \ge 1$. Lastly, let $x = t\ln(1+t)$ and notice that then $t \ge \frac{cx}{\ln(1+x)}$ for some constant $c>0$, if $t$ (or, equivalently, $x$) is large enough. In particular, $x\psi^{-1}(y) \le \frac{K}{c} \ln(1+x)\psi^{-1}(xy)$ for $x,y$ large enough and by adjusting the constant $\widehat{K} \ge \frac{K}{c}$ we obtain $x\psi^{-1}(y) \le \widehat{K} \ln(1+x)\psi^{-1}(xy)$ for any $x,y \ge \widehat{K}$.
\end{proof}

Our next lemma is a symmetrization inequality for Orlicz norms.

\begin{lemma}\label{le:symmetrization} If $X_1,\ldots,X_N$ are independent, $F$-valued, mean zero random variables, and $\varepsilon_1,\ldots,\varepsilon_N$ are independent Rademacher variables, independent of $X_1,\ldots,X_N$, then
\begin{displaymath}
  \frac{1}{2}\Big\|\sum_{i=1}^N \varepsilon_i X_i\Big\|_{\Psi} \le \Big\|\sum_{i=1}^N X_i\Big\|_{\Psi} \le 2 \Big\|\sum_{i=1}^N \varepsilon_i X_i\Big\|_{\Psi}.
\end{displaymath}
\end{lemma}
\begin{proof}
By the classical symmetrization inequalities (see, e.g., \cite[Lemma 6.3]{MR1102015}), for any $a > 0$,
\begin{displaymath}
\E \Psi\Big(\frac{\|\sum_{i=1}^N \varepsilon_i X_i\|}{2a}\Big) \le \E \Psi\Big(\frac{\|\sum_{i=1}^N X_i\|}{a}\Big) \le \E \Psi\Big(2 \frac{\|\sum_{i=1}^N \varepsilon_i X_i\|}{a}\Big),
\end{displaymath}
which easily implies the lemma.
\end{proof}

We will also need the following two easy lemmas concerning the relation between the first moment and the Orlicz norm of random variables.

\begin{lemma}\label{le:embedding}
There exists a constant $C_\Psi \in (0,\infty)$ such that for all $X \in L_\Psi(F)$, \begin{displaymath}
\|X\|_{1} \le C_\Psi \|X\|_\Psi.
\end{displaymath}
\end{lemma}
\begin{proof}
Due to homogenity, we can assume that $\|X\|_\Psi = 1$ (the case $\|X\|_\Psi = 0$ is trivial). Since $\Psi$ is convex and non-zero outside the origin, there exist constants $a_\Psi,b_\Psi > 0$ such that $\Psi(x) \ge a_\Psi x - b_\Psi$ for every $x \ge 0$. In particular, in light of the assumption $\|X\|_\Psi = 1$,
\begin{displaymath}
\mathbb E\|X\| \le \frac{1}{a_\Psi}\Big(\mathbb E\Psi(\|X\|) + b_\Psi \Big) \le \frac{1+b_\Psi}{a_\Psi} = \frac{1+b_\Psi}{a_\Psi}\|X\|_\Psi.
\end{displaymath}
\end{proof}

\begin{lemma}\label{le:mean-Orlicz}
For any $F$-valued integrable random variable $X$,
\begin{displaymath}
\| \E X\|_{\Psi} \le \max(\Psi(1),1)\|\E X\| \le \max(\Psi(1),1)\|X\|_1.
\end{displaymath}
\end{lemma}
\begin{proof}
The proof is an immediate consequence of convexity of $\Psi$. Indeed, \begin{displaymath}
\Psi\Big(\frac{\| \E X\| }{\|\E X\| \max(\Psi(1),1)}\Big) = \Psi\Big(\frac{1}{\max(\Psi(1),1)}\Big) \le 1,
\end{displaymath}
i.e., $\| \E X\|_{\Psi} \le \max(\Psi(1),1) \|\E X\|$. The second estimate follows from Jensen's inequality.
\end{proof}

Finally let us pass to the most important ingredient of the proof, which is a modification of Talagrand's inequality stated as the Basic Estimate (eq. (2.5)) in \cite{MR1048946}.

\begin{lemma}\label{le:crucial-lemma}
Let $X_1,\ldots,X_N$ be independent $F$-valued random variables and $\varepsilon_1,\ldots,\varepsilon_N$ be independent Rademacher variables, independent of the sequence $(X_i)_{i\le N}$. Let $M = \E \|\sum_{i=1}^N \varepsilon_i X_i\|$ and let $Y_1,\ldots,Y_N$ be the non-increasing rearrangement of $\|X_1\|,\ldots,\|X_N\|$. Then for any $q,k \in \mathbb N_+$ and $u,u' > 0$,
\begin{equation}\label{eq:crucial-inequality}
\p\Big( \Big\| \sum_{i \le N} \varepsilon_i X_i \Big\| \ge q^2M + u + u'\Big) \le \exp\Big(-\frac{u^2}{16q^3M^2}\Big) + \frac{4}{q^{k+1}} + \p\Big( \sum_{r \le k}Y_r \ge u'\Big).
\end{equation}
\end{lemma}

The proof of the above lemma is based heavily on the isoperimetric ideas due to Talagrand. We will provide the full argument for completeness. Let us note that a slightly weaker variant of Lemma \ref{le:crucial-lemma}, still sufficient for our purposes, could be derived from the original arguments from \cite{MR1048946}, starting with the isoperimetric inequality given in Theorem 7 therein, often referred to as \emph{approximation by $q$ points}. We however prefer to use instead a version of Theorem 7 stated in one of subsequent papers by Talagrand, i.e., \cite[Theorem 8.1]{MR1387624}, which is proved by simpler arguments and corresponds more closely to the way approximation by $q$ points has been presented since the publication of Talagrand's three seminal papers \cite{MR1361756,MR1419006,MR1387624} on concentration inequalities in the mid-90s.

Before we proceed, let us explain the main differences with the Basic Estimate \cite[eq. (2.5)]{MR1048946} of Talagrand.  In the Basic Estimate both $q^2$ and $q^3$ are replaced by $q$, however the inequality holds under a restriction $k \ge q$ and the second term on the right-hand side of \eqref{eq:crucial-inequality} is slightly worse. This is enough to provide a characterization of functions $\Psi(x) = e^{\psi(x)} - 1$, satisfying \eqref{eq:HJTal} under the assumption $\psi(x) = x\xi(x)$ for an increasing function $\xi$ as well as to obtain the inequality for other special cases considered in \cite{MR1048946}, i.e., for $\Psi(x) = x^p$ and $\Psi(x) = e^{x^\alpha} -1$. However, to provide a characterization in full generality, one needs to adjust $q$ and $k$ to the \emph{geometry} of $\Psi$ and it may happen that for the right choice of parameters $q,k$ one has $q \ge k$. At the same time, as will be seen in the proof of Theorem \ref{thm:characterization}, the worse exponents of $q$ in Lemma \ref{le:crucial-lemma} do not pose a problem in the argument.

Let us first introduce the main result by Talagrand, the proof of Lemma \ref{le:crucial-lemma} will be based on, namely \cite[Theorem 8.1]{MR1387624} (see also \cite[Theorem 4.12]{MR1849347}). Since Radon random variables can be approximated by variables taking only finitely many values, below we will not pay attention to measurability issues. Passing from a version of Lemma \ref{le:crucial-lemma} for discrete variables to the general one is standard. Below by $|\cdot|$ we denote the cardinality of a set.

\begin{theorem}\label{thm:q-point}
Let $(\Omega,\mathcal{F})$ be a measurable space and let $\p$ be a product probability measure on $\Omega^N$. For $A \subset \Omega^N$ and $\omega \in \Omega^N$ define
\begin{displaymath}
f_q(A,\omega) = \inf \{ | \{ i \le N : \omega_i \not \in \{x_i^1,...,x_i^q\} \}| : x^1,...,x^q \in A\}.
\end{displaymath}
Then for any integer $q \ge 2$,
\begin{displaymath}
\p( f_q(A,\cdot) \ge k) \le \frac{1}{q^k \p(A)^q}.
\end{displaymath}

\end{theorem}

In other words, if the set $A$ is relatively large, then for large $k$ with high probability we can cover all but at most $k$ of the coordinates of $x$ by coordinates from some ($x$-dependent) set of $q$ points from $A$. The strength of this result stems from the fact that the bound on probability does not deteriorate with growing $N$.

\begin{proof}[Proof of Lemma \ref{le:crucial-lemma}]
Without loss of generality we may assume that the variables $X_i$, $i\le N$, are defined on some product space $\Omega^{N}\times \Omega_\varepsilon$, equipped with a product measure $\p\otimes \p_\varepsilon$, where $\p$ is itself a product measure on $\Omega^N$, moreover $X_i$ depend only on the $i$-th coordinate, while $\varepsilon_i$ on the $(N+1)$-th coordinate. By $\E_\varepsilon$ we will denote integration with respect to $\p_\varepsilon$. An element of $\Omega^N$ will be denoted by $\omega = (\omega_1,\ldots,\omega_N)$.

Fix $q \in \mathbb N_+$ and define
\begin{displaymath}
A = \Big\{\omega \in \Omega^N\colon \E_{\varepsilon}\Big\|\sum_{i \le N} \varepsilon_i X_i(\omega) \Big\| \le qM\Big\}.
\end{displaymath}
By Markov's inequality,
\begin{align}\label{eq:Markov}
\p(A) \ge 1 - \frac{1}{q}.
\end{align}
Consider $F^\ast$ -- the dual of $F$. For simplicity we will denote the norm in $F^\ast$ also by $\|\cdot\|$.

For any element $\phi \in F^*$ with $\|\phi\| \le 1$ and $\omega \in A$, we get $ \E_{\varepsilon} \Big | \sum_{i \le N} \varepsilon_i \phi( X_i(\omega)) \Big| \le qM$ and so, by Khintchine's inequality (see, e.g., \cite{MR0430667} for the version with the optimal constant),

\begin{displaymath}
\sum_{i \le N}\phi^2(X_i(\omega)) = \E_\varepsilon \Big | \sum_{i \le N} \varepsilon_i \phi(X_i(\omega))\Big|^2 \le 2 \Big(\E_\varepsilon \Big | \sum_{i \le N} \varepsilon_i \phi(X_i(\omega))\Big|\Big)^2 
\le 2q^2M^2.
\end{displaymath}
Consider now any $k \in \mathbb N_+$ and $\omega \in H(A,q,k)$, where
\begin{displaymath}
H(A,q,k) = \{ \omega \in \Omega^N : \exists_{x^1,...,x^q \in A}\  | \{ i \le N : \omega_i \not \in \{x_i^1,...,x_i^q\}| \le k\}.
\end{displaymath}
By definition, we can find disjoint sets $I_\ell \subset \{i \le N : \omega_i = x_i^\ell\}$, $\ell = 1,\ldots, q$, and $J$ with $|J| \le k$ such that $I \cup J = \{1,...,N\}$, where $I = \bigcup_{\ell \le q}I_\ell$. Note that by Jensen's inequality and independence of $\varepsilon_i$'s, $I \mapsto \E_\varepsilon \| \sum_{i \in I} \varepsilon_i X_i(\omega)\|$ is an increasing function with respect to inclusion ordering on subsets $I$ of $\{1,...,N\}$. Since $x^\ell \in A$, this implies that
\begin{displaymath}
\E_\varepsilon \Big\| \sum_{i \in I_\ell} \varepsilon_i X_i(\omega)\Big\| = \E_\varepsilon \Big\|\sum_{i \in I_{\ell}} \varepsilon_i X_i(x^\ell)\Big\| \le \E_{\varepsilon} \Big\|\sum_{i \le N}\varepsilon_i X_i(x^\ell)\Big\| \le qM,
\end{displaymath}
and hence, by the triangle inequality,
\begin{displaymath}
\E_{\varepsilon} \Big\|\sum_{i \in I} \varepsilon_i X_i(\omega)\Big\| \le q^2M.
\end{displaymath}
Similarly, if $\phi \in F^*$ satisfies $\|\phi\| \le 1$, then
\begin{displaymath}
\sum_{i \in I_\ell} \phi^2 (X_i(\omega)) = \sum_{i \in I_{\ell}}\phi^2(X_i(x^\ell)) \le 2q^2M^2,
\end{displaymath}
so $\sum_{i \in I}\phi^2(X_i(\omega)) \le 2q^3M^2$.  Observe that $\sup_{\|\phi\|\le 1}(\sum_{i \in I}\phi^2(X_i(\omega)))^{1/2}$ is the Lipschitz constant of the convex function
$(\varepsilon_i)_{i \in I} \mapsto \|\sum_{i\in I} \varepsilon_i X_i(\omega)\|$. Therefore, by Theorem \ref{thm:convex-Ledoux}, for any $u > 0$,
\begin{displaymath}
\p_\varepsilon\Big( \Big\|\sum_{i \in I} \varepsilon_i X_i(\omega) \Big\| \ge q^2M + u\Big) \le \exp\Big( - \frac{u^2}{16q^3M^2}\Big).
\end{displaymath}

Note that due to $I \cup J = \{1,...,N\}, |J| \le k$ and the triangle inequality,
\begin{displaymath}
\Big\|\sum_{i \le N} \varepsilon_i X_i(\omega)\Big\| \le \Big\|\sum_{i \in I}\varepsilon_i X_i(\omega)\Big\| + \sum_{r \le k}Y_r(\omega),
\end{displaymath}
hence
\begin{displaymath}
\p_\varepsilon \Big( \Big\|\sum_{i \le N}\varepsilon_i X_i(\omega)\Big\| \ge q^2M + \sum_{r \le k}Y_r(\omega) + u\Big) \le \exp\Big( - \frac{u^2}{16q^3M^2}\Big)
\end{displaymath}
whenever $\omega \in H(A,q,k)$ and $u > 0$.
Since, $H(A,q,k)^c  = \{f_q(A,\cdot) \ge k+1\}$, Theorem \ref{thm:q-point} combined with \eqref{eq:Markov} gives
\begin{displaymath}
\p(H(A,q,k)^c) = \p( f_q(A,\cdot) \ge k+1) \le \frac{1}{q^{k+1} \p(A)^q} \le \frac{4}{q^{k+1}},
\end{displaymath}

The lemma follows now by Fubini's theorem and the union bound.
\end{proof}

\subsection{Proof of the main results\label{sec:proofs-main-results}}

Let us start with the proof of the necessity of the condition \eqref{eq:nsc}.

\begin{proof}[Proof of the first part of Theorem \ref{thm:characterization}]
Assume that $\Psi$ is an Orlicz function satisfying \eqref{eq:HJTal} for all sequences of i.i.d., real valued, symmetric, bounded random variables $X_1,\ldots,X_N$.

Let $u$ be such that $\Psi(u) > 1$. For $N \in \mathbb N_+$, consider a sequence $(X_j)_{j \le N}$ of i.i.d. random variables such that $\p  (X_j= \pm u) = (2N\Psi(u))^{-1}$, $\p  (X_j=0) = 1 - (N\Psi(u))^{-1}$. Clearly, $\max_{j \le N}|X_j| \le u$ and

\begin{displaymath}
\p  ( \max_{j \le N}|X_j| \neq 0) = \p\Big( \bigcup_{j=1}^N \{ |X_j| \neq 0\}\Big) \le N(1 - \p  (|X_1|=0)) = (\Psi(u))^{-1}.
\end{displaymath}

In particular,

\begin{displaymath}
\E \Psi\Big( \frac{\max_{j \le N}|X_j|}{\lambda}\Big) \le \Psi\Big(\frac{u}{\lambda}\Big) \frac{1}{\Psi(u)}\quad \text{for any}\ \lambda > 0,
\end{displaymath}
hence $\|\max_{j \le N}|X_j|\|_{\Psi} \le 1$. Moreover, by the triangle inequality, $\|\sum_{j=1}^N X_j\|_{1} \le \frac{u}{\Psi(u)}$. Since $\Psi(u) \ge cu$ for some $c>0$ and all $u$ large enough, we have $\frac{u}{\Psi(u)} \le \frac{1}{c}$. It follows from \eqref{eq:HJTal} that $\|\sum_{j=1}^N X_j\|_{\Psi} \le (1+\frac{1}{c})D$ if only $u$ is large enough.

Hence, there is some constant $C=C_\Psi > 1$ such that $\|\sum_{j=1}^N X_j \|_{\Psi} \le C$ for $u$ large enough. Now, observe that
\begin{displaymath}
1 \ge \E \Psi \Big( \frac{1}{C}\Big|\sum_{j=1}^NX_j\Big|\Big) \ge \Psi\Big(\frac{N u}{C}\Big)\p  \Big( \Big|\sum_{j=1}^N X_j\Big| = Nu\Big) \ge \Psi\Big(\frac{N u}{C}\Big) (2N\Psi(u))^{-N}.
\end{displaymath}
 It follows that for $u$ large enough and all $N \in \mathbb{N}_+$, we get $\Psi( \frac{Nu}{C}) \le (2N\Psi(u))^N$. Now, recall $\Psi(u) = e^{\psi(u)} - 1$, so that for some constant $a \in (0,1)$ and all $u$ large enough, $e^{a\psi(u)} \le \Psi(u)$. As a consequence,
 \begin{displaymath}
 a \psi\Big(\frac{Nu}{C}\Big) \le N\ln(2N) + N\psi(u).
\end{displaymath}
For $t$ of the form $t=\frac{N}{C}$, where $N \in \mathbb{N}_+$, and all $u$ large enough we obtain
\begin{displaymath}
\psi(tu) \le \frac{C}{a}t\ln(1+2tC) + \frac{C}{a}t\psi(u) \le \frac{C(1+\ln(1+2C))}{a}\Big( t\ln(1+t) + t\psi(u)\Big).
\end{displaymath}
Given any $u,s$ large enough, we can find $t = \frac{N}{C}$ such that $s\le t<s+1$ (recall that $C > 1$). Hence, $$ \psi(su) \le \tilde{C}\Big(2s\ln(1+2s) + 2s\psi(u)\Big) \le 2\tilde{C}(1+\ln(3))\Big(s\ln(1+s) + s\psi(u)\Big)$$ for some constant $\tilde{C}$ depending only on $\Psi$. It suffices to set $K=2\tilde{C}(1+\ln(3))$.
\end{proof}

We now proceed with the much more involved proof of sufficiency of the condition \eqref{eq:nsc}.

\begin{proof}[Proof of the second part of Theorem \ref{thm:characterization}]

In this proof by $C,c$ we shall denote positive constants depending only on $\Psi$ which may vary in value between occurrences, even within the same line.

Assume that $\Psi$ satisfies the condition \eqref{eq:nsc}. Denote
\begin{displaymath}
Z = \Big \|\sum_{j=1}^N X_j\Big\|.
\end{displaymath}

We will start by a standard reduction to the case of symmetric random variables, which we will present for the sake of completeness.
First, by the triangle inequality and Lemma \ref{le:mean-Orlicz},
\begin{displaymath}
\| Z\|_{\Psi} \le \Big \| \sum_{i=1}^N (X_i- \E X_i)\Big\|_{\Psi} + C\|Z\|_1.
\end{displaymath}
If we assume that \eqref{eq:HJTal} is satisfied for symmetric random variables then, using Lemma \ref{le:symmetrization}, we further get
\begin{multline*}
\|Z\|_{\Psi} \le 2 \Big\|\sum_{i=1}^N \varepsilon_i(X_i-\E X_i)\Big\|_{\Psi} + C\|Z\|_1 \\
\le 2D\Big(\Big\|\sum_{i=1}^N \varepsilon_i(X_i-\E X_i)\Big\|_1 + \Big\|\max_{i\le N} \|X_i-\E X_i\| \Big\|_{\Psi}\Big) + C\|Z\|_1\\
\le  2D\Big(2 \Big\|\sum_{i=1}^N (X_i-\E X_i)\Big\|_1 + \Big\|\max_{i\le N} \|X_i-\E X_i\| \Big\|_{\Psi}\Big) + C\|Z\|_1\\
\le  (8D+C)\|Z\|_1 + 4D\Big\|\max_{i\le N} \|X_i\| \Big\|_{\Psi},
\end{multline*}
where in the last estimate we again used the triangle inequality and $\|\E X_i\|_{\Psi} \le \|X_i\|_{\Psi}$, which follows from Jensen's inequality.
In the remaining part of the proof we will thus assume that $X_1,\ldots,X_N$ are symmetric.
By homogenity we can also assume that
\begin{align}\label{eq:normalization-assumption}
\|Z\|_{1} + \Big\|\max_{j \le N}\|X_j\| \Big\|_{\Psi} = 1.
\end{align}

Hence, the theorem will be proven if we show that
\begin{align}\label{eq:goal}
\E \Psi\Big(\frac{Z}{\rho}\Big) \le 1
\end{align}
for some constant $\rho=\rho_\Psi>0$. Consider a large constant $a$, the value of which will be set later on ($a$ will depend on $\Psi$ only). Our goal is to bound
\begin{displaymath}
  \E \Psi\Big(\frac{Z}{3a}\Big) = \int_0^\infty \p\Big( \Psi\Big(\frac{Z}{3a}\Big) \ge t\Big)dt
\end{displaymath}
by a finite constant, depending only $\Psi$.

We have
\begin{displaymath}
\p\Big( \Psi\Big(\frac{Z}{3a}\Big) \ge t\Big) = \p\Big( Z \ge 3a\Psi^{-1}(t)\Big) = \p\Big( Z \ge 3a\psi^{-1}(\ln(1+t))\Big).
\end{displaymath}

To estimate the above expression we will use Lemma \ref{le:crucial-lemma} with an appropriate choice of parameters, depending on $t$. We will work under the assumption that $t > t_\Psi$ for a suitable choice of $t_\Psi$, depending only on $\Psi$. The requirements on $t_\Psi$ will become clear in the course of the proof.

In order to make the argument more transparent let us adopt another convention regarding constants and parameters in the rest of the proof and provide a short glimpse into its structure. We will encounter three parameters, depending only on $\Psi$: the parameter $a$ we are after and auxiliary parameters $R, p_0$, which we will use to show the existence of $a$ (their role will be explained later on). We will still use $C$, $c$ for respectively large and small positive constants depending only on $\Psi$. As already mentioned, the values of $C,c$ may change between occurrences. The additional convention we introduce now is that the constants $C,c$ will not depend on the parameters $a,R,p_0$. The inequalities involving those constants will hold for $t > t_\Psi$ and $t_\Psi$ may depend on $a,R,p_0$, it is only important that $C,c$ do not. We are going to first define $R$ (in \eqref{eq:assumption-on-R}), then $p_0$ (in \eqref{eq:assumptions-on-p-naught}), and finally $a$ (in \eqref{eq:assumptions-on-a}). In particular $p_0$ and $a$ will be defined with use of the constants $C,c$.

Let
\begin{align}\label{eq:q-k-definition}
q(t) = \Big\lfloor \Big(\frac{(\psi^{-1}(\ln(1+t)))}{\ln \ln(1+t)}\Big)^{\frac{1}{3}} \Big\rfloor \textrm{ and } k(t) = 4 \Big\lceil\frac{\ln(1+t)}{\ln q(t)}\Big \rceil.
\end{align}
Note that by Lemma \ref{le:phi-inverse}, for $s$ large enough,
\begin{align}\label{eq:bound-on-phi-inverse}
\psi^{-1}(s) \ge c \frac{s}{\ln s}
\end{align}
and as a consequence, for $t> t_\Psi$,
\begin{align}\label{eq:lower-bound-on-q}
q(t) \ge c \frac{(\ln t)^{1/3}}{\ln \ln t} \ge 4
\end{align}
and
\begin{align}\label{eq:log-of-q}
\ln q(t) \ge c \ln \psi^{-1}(\ln(1+t)).
\end{align}

The inequality $q^2 \le a\psi^{-1}(\ln(1+t))$, which holds for $t > t_\Psi$, together with \eqref{eq:normalization-assumption} imply that $q^2 \E Z \le a\psi^{-1}(\ln(1+t))$, and so, by Lemma \ref{le:crucial-lemma} applied with $u = u' = a\psi^{-1}(\ln(1+t))$, one gets for $t > t_\Psi$,
\begin{multline}\label{eq:q-points-inequality}
\p  \Big( \Psi\Big(\frac{Z}{3a}\Big) \ge t\Big) = \p   \Big( Z \ge 3a\Psi^{-1}(t)\Big) = \p\Big( Z \ge 3a\psi^{-1}(\ln(1+t)) \Big)\\
\le \exp\Big(-\frac{a^2}{16q(t)^3} (\psi^{-1}(\ln(1+t)) )^2\Big)  + e^{-k(t)\ln(q(t))} + \p\Big( \sum_{r \le k(t)} Y_r > a\psi^{-1}(\ln(1+t)) \Big)\\
\le \exp(-c a^2\ln(1+t)) + \exp(-4 \ln(1+t)) + \p\Big( \sum_{r \le k(t)} Y_r > a\psi^{-1}(\ln(1+t)) \Big),
\end{multline}
where we used that by the definition of $q(t)$ and \eqref{eq:bound-on-phi-inverse},
\begin{displaymath}
  \frac{(\psi^{-1}(\ln(1+t)) )^2}{q(t)^{3}} \ge \psi^{-1}(\ln(1+t))\ln \ln (1+t) \ge c \ln(1+t).
\end{displaymath}

The first and second terms on the right-hand side of \eqref{eq:q-points-inequality} are integrable, provided $a$ is large enough. It thus remains to bound the last term, i.e., $\p  ( \sum_{r \le k(t)}Y_r > a\psi^{-1}(\ln(1+t)))$ by an integrable function.

Let $m$ be the largest integer such that $2^m \le k(t)$, so that $m \le C \ln k(t)$. Let $p\le m$ be the largest nonnegative integer such that $Y_{2^p} \ge R$, where $R=R_\Psi$ is a positive constant to be determined later. If such an integer does not exist, we set $p=-1$. Note that $p$ is a random variable.

If $a$ is large enough with respect to $R$, then
\begin{align}\label{eq:1st-condition-on-a}
 Y_{2^{p+1}}+...+Y_{k(t)} \le R k(t) \le CR \frac{\ln t}{\ln \ln t} \le \frac{a}{2} \psi^{-1}(\ln(t+1)),
\end{align}
where in the second inequality we used \eqref{eq:q-k-definition} and \eqref{eq:lower-bound-on-q}, while in the last one the inequality \eqref{eq:bound-on-phi-inverse}.

It follows that on the event
\begin{align}\label{eq:definition-of-A-t}
A_t :=\Big\{ \sum_{r \le k(t)} Y_r \ge a \psi^{-1}(\ln(t+1))\Big\}
\end{align}
we have $p \ge 0$ and
\begin{displaymath}
Y_1 + ... +Y_{2^{p+1}} \ge \frac{a}{2}\psi^{-1}(\ln(t+1)).
\end{displaymath}

Let us now fix a constant $p_0 \in \mathbb N_+$ also to be determined later and set
\begin{displaymath}
B_t := \Big\{\sum_{r \le 2^{p_0}}Y_r \ge \frac{a}{4}\psi^{-1}(\ln(1+t))\Big\}.
\end{displaymath}
Note that
\begin{displaymath}
  B_t \subset \Big\{ \max_{j \le N}\|X_j\| \ge \frac{a}{4 \cdot 2^{p_0}}\psi^{-1}(\ln(1+t))\Big\} = \Big\{ \Psi\Big( \frac{4\cdot 2^{p_0}}{a} \max_{j \le N}\|X_j\|\Big) \ge t \Big\}.
\end{displaymath}
Hence, by \eqref{eq:normalization-assumption}, for $a > 4\cdot 2^{p_0}$,
\begin{align}\label{eq:bound-on-B_t}
\int_0^\infty \p  (B_t)dt \le 1.
\end{align}

We shall now estimate $\p  (A_t \cap B_t^c)$. On this set, we have
\begin{align}\label{eq:integral-with-set-B}
\sum_{i=2^{p_0}+1}^{2^{p+1}}Y_{i} \ge \frac{a}{4}\psi^{-1}(\ln(1+t)),
\end{align}
so due to monotonicity of $(Y_r)_r$,
\begin{align}\label{eq:powers-of-two}
\sum_{i=p_0}^{p}2^iY_{2^i} \ge \sum_{i=p_0}^{p} \sum_{r=2^i+1}^{2^{i+1}}Y_r = \sum_{i=2^{p_0}+1}^{2^{p+1}}Y_i \ge \frac{a}{4}\psi^{-1}(\ln(1+t)).
\end{align}

Note that this implies in particular that on $A_t\cap B_t^c$ the above sums are not empty (which can in principle happen outside this event).

It follows that there exist positive integers $m(p_0),...,m(p)$ such that setting $a_i = 2^{m(i)}$ for $i\in\{p_0,...,p\}$ we have
\begin{align}\label{eq:restrictions}
\begin{cases} \sum_{i=p_0}^p a_i 2^i \ge \frac{a}{8}\psi^{-1}(\ln(1+t)), \\
a_i \ge \frac{R}{2},\ \\
Y_{2^i} \ge a_i, \\
m(i) \le \log_2 (\frac{a}{8}\psi^{-1}(\ln(1+t))) \le C\ln(a \psi^{-1}(\ln(1+t))).\end{cases}
\end{align}

Indeed, one can take $m(i)$ such that $2^{m(i)} \le \min(Y_{2^i},\frac{a}{8}\psi^{-1}(\ln(1+t))) < 2^{m(i) + 1}$. The first inequality of \eqref{eq:restrictions} follows then from \eqref{eq:powers-of-two}. By the definitions of $p$ and $m(i)$ the second inequality holds for $t > t_\Psi$, provided that $t_\Psi$ is sufficiently large with respect to $R$ (recall that we are going to choose $R$ depending only on $\Psi$, in particular independent of $t$). The remaining inequalities follow directly from the definition of $m(i)$.

Thus, by the union bound,
\begin{align}\label{eq:union-bound}
\p(A_t\cap B_t^c) \le \sum \p(\forall_{p_0 \le i \le p}\, Y_{2^i} \ge a_i),
\end{align}
where the summation is over all choices of numbers $p,a_{p_0},...,a_p$, satisfying \eqref{eq:restrictions} (note that we slightly abuse the notation and treat now $p \in \{p_0,\ldots,m\}$ as a parameter, rather than a random variable).

The number of such choices is at most
\begin{displaymath}
m \cdot \Big( C\ln(a\psi^{-1}(\ln(1+t))\Big)^m = \exp\Big( \ln m + m \ln \bigr(C \ln \bigr(a \psi^{-1}\bigr(\ln(1+t)\bigr)\bigr)\bigr)\Big).
\end{displaymath}
Recall that by the convexity of $\Psi$, we have $\psi(t) \ge \ln(1+ct)$ for some $c>0$ and $t$ large enough. Letting $t = \psi^{-1}(s)$ we get $s \ge \ln(1+c\psi^{-1}(s))$ and, in particular, $\frac{1}{c}(e^s-1) \ge \psi^{-1}(s)$. As a consequence,
\begin{displaymath}C \ln(a \psi^{-1}(\ln(1+t))) \le C \ln\Big(\frac{a}{c}t\Big) \le C\ln(\frac{a}{c}) + C\ln t,
\end{displaymath}
hence for $t \ge \frac{a}{c}$, the number of choices of $p,a_{p_0},...,a_p$, satisfying \eqref{eq:restrictions}, is at most
\begin{align}\label{eq:number-of-choices}
\exp\Big( \ln m + mC \ln \ln (t+1)\Big) \le \exp \Big( C (\ln \ln (t+1))^2\Big),
\end{align}
where in the last inequality we used the estimate $m \le \log_2 k(t) \le C\ln\ln(t+1)$ which follows from the definition of $m$ and \eqref{eq:q-k-definition}.

Let us now estimate the probability that for fixed $p_0 \le p \le m$ and $m(p_0),...,m(p)$ as above, we have $Y_{2^i} \ge a_i$ for every $i \in \{p_0,...,p\}$. We may assume that $p_0 \ge 2$. Using Lemma \ref{le:Orlicz-Chebyshev} we may now choose $R=R_\Psi$ large enough so that $R \ge 2$ and
\begin{align}\label{eq:assumption-on-R}
\p  ( \max_{j \le N}\|X_j\| \ge R/2)\le 2e^{-\psi(R/2)} \le \frac{1}{16}.
\end{align}
Then
\begin{multline*}
\p  (\forall_{p_0 \le i \le p}:\  Y_{2^i} \ge a_i) \le \sum_{\substack{I_{p_0},...,I_p \subset \{1,...,N\} \\ (I_{i})_{p_0 \le i \le p}\  \text{pairwise disjoint} \\ |I_i| = 2^{i-1},\  p_0 \le i \le p}} \prod_{i=p_0}^p \prod_{j \in I_i}\p  (\|X_j\| \ge a_i) \\
\le \sum_{\substack{I_{p_0},...,I_p \subset \{1,...,N\} \\ |I_i| = 2^{i-1},\  p_0 \le i \le p}} \prod_{i=p_0}^p \prod_{j \in I_i}\p  (\|X_j\| \ge a_i) \le \prod_{i=p_0}^p \Big( \sum_{j \le N}\p  (\|X_j\| \ge a_i)\Big)^{2^{i-1}}.
\end{multline*}

Invoking Lemma 3.2 and \eqref{eq:assumption-on-R} together with the trivial inequality $x \le \frac{1}{4}\sqrt{x}$ which holds for every $x \in [0,1/16]$, we may bound the right-hand side of the above inequality by
\begin{displaymath} \prod_{i=p_0}^p \Big(2 \p  (\max_{j \le N}\|X_j\| \ge a_i) \Big)^{2^{i-1}} \le \prod_{i=p_0}^p \Big( 2 \cdot \frac{1}{4} \p  ^{\frac{1}{2}}(\max_{j \le N}\|X_j\| \ge a_i)\Big)^{2^{i-1}}.
\end{displaymath}

By \eqref{eq:normalization-assumption} and Lemma \ref{le:Orlicz-Chebyshev}, $\p  (\max_{j \le N}\|X_j\| \ge t) \le 2e^{-\psi(t)}$, hence if $t \ge t_a$, then
\begin{multline}\label{eq:bound-on-intersection}\p  (\forall_{p_0 \le i \le p}\, Y_{2^i} \ge a_i) \le \prod_{i=p_0}^p \Big( e^{-\frac{1}{2}\psi(a_i)}\Big)^{2^{i-1}} = \exp\Big( -\sum_{i=p_0}^p 2^{i-2}\psi(a_i)\Big)\\
= \exp\Big( - 2^{p_0-2} \sum_{i=p_0}^p 2^{i-p_0} \frac{a_i}{\sum_{j=p_0}^p 2^{j-p_0}a_j} \frac{\sum_{j=p_0}^p 2^{j-p_0}a_j}{a_i} \psi(a_i) \Big) =: (*).
\end{multline}

Note that $ \frac{1}{a_i}\sum_{j=p_0}^p 2^{j-p_0}a_j \ge 1$ and, moreover, $\frac{1}{a_i}\sum_{j=p_0}^p 2^{j}a_j \le \max_j\{a_j\} 2^{p+1} = 2^{p+1} 2^{\max\{m(j)\}}$. Recall that
\begin{displaymath}
m(j) \le C\ln(a \psi^{-1}(\ln(1+t))) \le 2C\ln(\psi^{-1}(\ln(1+t)))
\end{displaymath}
if $t$ is large enough, say $t \ge t_a$, where $t_a$ depends only on $a$.
Thus for $t \ge t_a$ we have $2^{p+1} 2^{\max\{m(j)\}} \le 2k(t) (\psi^{-1}(\ln(1+t)))^C$, and
\begin{align}\label{eq:auxiliary-estimate}
\frac{1}{a_i}\sum_{j=p_0}^p 2^{j-p_0}a_j \le \frac{1}{2^{p_0}}k(t)( \psi^{-1}(\ln(1+t)))^C \le (\psi^{-1}(\ln(1+t)))^{C+2},
\end{align}
where the last inequality is a crude estimate based on \eqref{eq:q-k-definition}, \eqref{eq:bound-on-phi-inverse} and \eqref{eq:lower-bound-on-q}.

Note that by adjusting the constant $K$ in the condition \eqref{eq:nsc} we may assume that it holds for any $s,u \ge 1$ (for simplicity we will call the adjusted constant also by $K$). Using this inequality with $s = \frac{1}{a_i}\sum_{j=p_0}^p 2^{j-p_0}a_j \ge 1$ and $u = a_i \ge R/2 \ge 1$, we obtain
\begin{multline*} (*) \le \exp\Big(-\frac{2^{p_0-2}}{K} \Big( \psi\Big(\sum_{j=p_0}^p 2^{j-p_0}a_j\Big) - K \sum_{i=p_0}^p 2^{i-p_0}\ln\Big(1+ \frac{1}{a_i}\sum_{j=p_0}^p 2^{j-p_0}a_j\Big)\Big)\Big)\\
\le \exp\Big( - \frac{2^{p_0-2}}{K}\Big( \psi\Big(\frac{a}{8\cdot 2^{p_0}}\psi^{-1}(\ln(1+t))\Big) - \frac{K}{2^{p_0-1}}k(t) \ln(1 + (\psi^{-1}(\ln(1+t)))^{C+2})\Big) \Big) \\
\le \exp\Big( - \frac{2^{p_0-2}}{K}\Big( \psi\Big(\frac{a}{8 \cdot 2^{p_0}}\psi^{-1}(\ln(1+t))\Big) - \frac{K(C+2)}{2^{p_0-1}}\frac{\ln(1+t)}{c\ln(\psi^{-1}(\ln(1+t)))}\ln(\psi^{-1}(\ln(1+t))) \Big) \Big),
\end{multline*}
where in the second inequality we also used $2^p \le 2^m \le k(t)$, the inequality \eqref{eq:auxiliary-estimate}, and the first condition in \eqref{eq:restrictions}, while in the third inequality the definition \eqref{eq:q-k-definition} and the estimate \eqref{eq:log-of-q}.

Going back to \eqref{eq:bound-on-intersection} we obtain that, if $p_0$ is large enough so that
\begin{align}\label{eq:assumptions-on-p-naught}
\frac{K(C+2)}{c2^{p_0-1}} < \frac{1}{2} \textrm{ and } \frac{2^{p_0-2}}{K} \ge 8,
\end{align}
then for all $a \ge 8 \cdot 2^{p_0}$ and $t \ge t_a$,
\begin{displaymath}
\p(\forall_{p_0 \le i \le p}:\ Y_{2^i} \ge a_i)\le \exp\Big(-8 \Big( \psi(\psi^{-1}(\ln(1+t))) - \frac{1}{2}\ln(1+t)\Big)\Big) = \exp\Big(-4\ln(1+t)\Big).
\end{displaymath}
Note that the choice of $p_0$ and $a$ can be made in such a way that they depend only on $\Psi$, as a consequence $t_a$ also depends only on $\Psi$.

Taking into account the bound \eqref{eq:number-of-choices} on the number of choices of $p,a_{p_0},...,a_{p}$ and \eqref{eq:union-bound}, we get
\begin{displaymath}
\p(A_t \cap B_t^c) \le \exp(C(\ln \ln (1+t))^2 - 4\ln(1+t)) \le \exp(-3\ln(1+t)) = \frac{1}{(1+t)^3}
\end{displaymath}
for $t\ge t_\Psi$. This means that for some constant $\kappa_\Psi$, depending only on $\Psi$,
\begin{displaymath}\int_0^\infty\p  (A_t)dt \le \int_0^\infty \p  (A_t \cap B_t^c) dt + \int_0^\infty \p  (B_t)dt \le  \kappa_\Psi
\end{displaymath}
where we used \eqref{eq:bound-on-B_t}.
Combining this with the definition \eqref{eq:definition-of-A-t} of the event $A_t$ and \eqref{eq:q-points-inequality} we obtain that
for
\begin{align}\label{eq:assumptions-on-a}
a \ge \max(8\cdot 2^{p_0},CR),
\end{align}
we have
\begin{displaymath}
\E \Psi\Big(\frac{Z}{3a}\Big)  = \int_0^\infty \p  \Big( \Psi\Big(\frac{Z}{3a}\Big) \ge t\Big)dt  < \tilde{\kappa}_{\Psi},
\end{displaymath}
where $\tilde{\kappa}_\Psi$ is another constant, depending only on $\Psi$.
Recall that the condition $a \ge 8\cdot2^{p_0}$ was used in order to get integrability of the function $t\mapsto \p(A_t)$, in turn the condition $a \ge CR$ was used in \eqref{eq:1st-condition-on-a}. The inequality $R \ge 2$ ensures also (for $C$ sufficiently large) the integrability of the first summand on the right-hand side of \eqref{eq:q-points-inequality}.

By convexity of $\Psi$ the above estimate implies \eqref{eq:goal} for some $\rho$, depending only on $a$ and $C$ and thus only on $\Psi$ (one may take $\rho = 3a\max(\tilde{\kappa}_\Psi,1)$). This ends the proof of the theorem.
\end{proof}

Let us now pass to the proof of Proposition \ref{prop:counterexample}.

\begin{proof}[Proof of Proposition \ref{prop:counterexample}]

For a convex function $F$ by $F'(x+)$ and $F'(x-)$ we will denote respectively its right and left derivatives at $x$.

Note that by the convexity of $\Phi$, $\Phi'(x+) \ge \frac{\Phi(x)-\Phi(0)}{x-0} = \frac{\Phi(x)}{x}$, hence we also have
\begin{align}\label{eq:superpolynomial}
\liminf_{x \to +\infty} \frac{\Phi'(x+)}{x^p} = \infty
\end{align}
for all $p>0$.
By Remark \eqref{re:nsc1} the condition \eqref{eq:nsc} is equivalent to the inequality
\begin{align}\label{eq:HJ-in-terms-of-Psi}
\Psi(tu) \le K' t^{K't} \Psi(u)^{K't}
\end{align}
for some constant $K'>0$ and all $t,u \ge K'$.

To prove the proposition it is therefore enough to construct an Orlicz function $\Psi$, dominated by $\Phi$ far away from zero, together with a sequence of numbers $u_1,u_2,\ldots$ such that $u_n\to \infty$ and
\begin{align}\label{eq:to-fulfil}
\Psi(nu_n) > n n^{n^2} \Psi(u_n)^{n^2}
\end{align}
for large $n$.

We shall define the sequence $u_n$, inductively, in such a way that $u_{n+1} \ge nu_n$. Along the construction we will also define a function $\Psi$, affine on $[u_n,u_{n+1}]$ and satisfying \eqref{eq:to-fulfil}.

Set first $u_1 = 0$ and let $u_2$ be any number greater than 1, such that
\begin{displaymath}
  \Phi(u_2) > 1, \; \Phi'(u_2+) > \frac{1}{u_2}, \; \Phi'(u_2+)u_2 + 1 > 2 \cdot 2^{2^2}.
\end{displaymath}

Define then $\Psi(x) = \frac{x}{u_2}$ on $[u_1,u_2]$. In terms of $\Psi$ the above inequalities translate into
\begin{displaymath}
  \Phi(u_2)> \Psi(u_2), \; \Phi'(u_2+) > \Psi'(u_2-),\; \Phi'(u_2+)(2u_2 - u_2) + 1 > 2 \cdot 2^{2^2} \Psi(u_2)^{2^2}.
\end{displaymath}

Assume now that for some $n\ge 2$, we have already defined a sequence $0=u_1 < 1 < u_2< \ldots<u_n$, and a convex increasing function $\Psi\colon[u_1,u_n] \to [0,\infty)$, such that $\Psi \le \Phi$ on $[u_2,u_n]$ and
\begin{itemize}
\item[(i)] $\Psi$ is affine on $[u_{k-1},u_{k}]$ for $k=2,\ldots,n$;
\item[(ii)] $ku_k < u_{k+1}$ for $k=2,\ldots,n-1$;
\item[(iii)] $\Psi(ku_{k}) > k\cdot k^{k^2}\Psi(u_k)^{k^2}$ for $k=2,\ldots,n-1$;
\item[(iv)] $\Phi'(u_k+) > \Psi'(u_k-)$ for $k=2,\ldots,n$;
\item[(v)] $\Phi'(u_n+)(nu_n - u_n) + \Psi(u_n) > n\cdot n^{n^2} \Psi(u_n)^{n^2}$.
\end{itemize}
Note that for $n=2$ some of the conditions above are empty.

We can then find a number $u_{n+1} > nu_n$ such that if for every $x \in (u_n,u_{n+1}]$ we define $\Psi(x) = \Phi'(u_n+)(x-u_n) + \Psi(u_n)$, then $\Psi$ is convex increasing on $[u_1,u_{n+1}]$, $\Psi \le \Phi$ on $[u_2,u_{n+1}]$ and moreover all the above conditions hold for $n+1$ instead of $n$. It is enough to find $u_{n+1}$ such that $u_{n+1} > nu_n$ and
\begin{displaymath}
\Phi'(u_{n+1}+) > (n+1)\cdot (n+1)^{(n+1)^2} \Big(\Phi'(u_n+)(u_{n+1}-u_n) + \Psi(u_n)\Big)^{(n+1)^2} + \Phi'(u_n+).
\end{displaymath}
which is possible by \eqref{eq:superpolynomial}, since the right hand side above is a polynomial function of $u_{n+1}$.
Indeed, the above inequality together with the condition $u_{n+1} > 1$,  implies the inequality (v) with $n$ replaced by $n+1$.

Together with the induction assumption and the definition of $\Psi$ on $(u_n,u_{n+1}]$ (v) for $n+1$ implies (iii) with $n+1$ instead of $n$. Since $\Psi'(u_{n+1}-) = \Phi'(u_n+) < \Phi'(u_{n+1}+)$, we also get (iv) with $n+1$ instead of $n$. The conditions (i) and (ii) are trivial by the construction.

Now, by the induction assumption (iv), we have $\Psi'(x+) \ge \Psi(u_n-)$ for $x \in [u_n,u_{n+1})$, which together with convexity and monotonicity of $\Psi$ on $[u_1,u_n]$ gives convexity and monotonicity on $[u_1,u_{n+1}]$. Moreover, the convexity of $\Phi$, together with $\Psi(u_n) \le \Phi(u_n)$ implies $\Psi \le \Phi$ on $[u_n,u_{n+1}]$ and thus, by the induction assumption, on $[u_2,u_{n+1}]$.

Thus, the inductive construction produces an Orlicz function $\Psi$ such that $\Psi(x) \le \Phi(x)$, together with a sequence $u_n\to \infty$ such that \eqref{eq:to-fulfil} holds, which ends the proof.
\end{proof}

\subsection{Proof of results from Section \ref{sec:weak-variance}\label{sec:proofs-weak-variance}}

Similarly as in \cite{MR2424985} in the $\Psi_\alpha$ case, we will combine the inequality by Talagrand for bounded empirical processes with a truncation argument based on the Hoffmann-Jørgensen inequality.

\begin{proof}[Proof of Theorem \ref{thm:weak-variance}]
As noted in Remark \ref{re:Bennett-Bernstein}, it is enough to prove \eqref{eq:Bennett}.  We can assume that the class $\mathcal{A}$ is finite. In what follows by $C$ we will denote constants which depend only on $\Psi$. Their value may change between occurrences, even within the same line.

For $x \in \mathcal{S}$ let $F(x) = \sup_{f\in \mathcal{A}} |f(x)|$ and $R = 8 \E\max_{i\le N} F(X_i)$. Denote
\begin{displaymath}
  S_1 = \sup_{f\in \mathcal{A}} \sum_{i=1}^N f(X_i)\ind{F(X_i) \le R}, \; S_2 = \sup_{f\in \mathcal{A}} \Big| \sum_{i=1}^N f(X_i)\ind{F(X_i) > R}\Big|.
\end{displaymath}

We have
\begin{align}\label{eq:truncation}
|S - S_1| \le S_2,\; |\E S - \E S_1| \le \E |S - S_1| \le \E S_2.
\end{align}

Observe that by Markov's inequality and the definition of $R$,
\begin{displaymath}
  \p\Big(\max_{n\le N} \sup_{f\in \mathcal{A}} \Big|\sum_{i=1}^n f(X_i)\ind{F(X_i) > R}\Big| > 0\Big) \le \p(\max_{n \le N} F(X_n) > R) \le \frac{1}{8},
\end{displaymath}
so by the original Hoffmann-Jørgensen inequality \eqref{eq:HJ-quantile} with $p=1$ and $t_0 = 0$, together with Lemma \ref{le:embedding}, we have
\begin{displaymath}
\E S_2 \le C\E \max_{i\le n} F(X_i) \le C \|\max_{i\le N} F(X_i)\|_{\Psi}
\end{displaymath}
(formally, we apply \eqref{eq:HJ-quantile} to the finite-dimensional Banach space $\ell_\infty(\mathcal{A})$ of bounded functions on $\mathcal{A}$ and random vectors
$\widehat{X}_i = (f(X_i)\ind{F(X_i) > R})_{f\in \mathcal{A}}$).

By Theorem \ref{thm:characterization}, $\Psi$ satisfies \eqref{eq:HJTal}, so we also obtain
\begin{multline*}
  \|S_2 + \E S_2\|_\Psi
  \le  \|S_2\|_{\Psi} + \|\E S_2\|_{\Psi} \le D(\E S_2 + \|\max_{i\le N} F(X_i)\|_{\Psi}) + \|\E S_2\|_{\Psi}\\
  \le C \|\max_{i\le N} F(X_i)\|_{\Psi} = CU.
\end{multline*}
where we also used that $\| \E S_2\|_{\Psi} \le C\E S_2$ (see Lemma \ref{le:mean-Orlicz}).

Thus, by Lemma \ref{le:Orlicz-Chebyshev} we get
\begin{align}\label{eq:unbounded-part-tail}
\p(S_2 + \E S_2 \ge t) \le \frac{2}{\Psi\big(\frac{t}{CU}\big)+1}
\end{align}
for all $t > 0$.

Now, Theorem \ref{thm:Talagrand-inequality} applied to the class $\widehat{A} = \{f(\cdot)\ind{F(\cdot) \le R}\colon f\in \mathcal{A}\}$, implies that for a sufficiently large constant $C$ and all $t > 0$,
\begin{multline}\label{eq:bounded-part-tail}
\p(|S_1 - \E S_1| \ge t) \le C\exp\Big(-\frac{t}{C R}\log\Big(1+\frac{tR}{\E\sup_{f\in A} \sum_{i=1}^N f(X_i)^2\ind{F(X_i)\le R}}\Big)\Big)\\
\le C \exp\Big(-\frac{t}{C^2 U}\log\Big(1+\frac{C tU}{\Sigma^2}\Big)\Big)\\
\le C \exp\Big(-\frac{t}{C^2 U}\log\Big(1+\frac{tU}{\Sigma^2}\Big)\Big),
\end{multline}
where in the second inequality we used the fact that $R \le C U$ and that for $a > 0$ the function
\begin{displaymath}
  R \mapsto \frac{a}{R}\log\big(1+aR)
\end{displaymath}
is decreasing on $(0,\infty)$. Indeed, differentiating with respect to $R$ gives
\begin{displaymath}
  \frac{a}{R^2} \Big(\frac{aR}{aR+1} - \log(1+aR)\Big) < 0
\end{displaymath}
thanks to the well known inequality $\log(1+x) \ge \frac{x}{x+1}$.

Note now that due to \eqref{eq:truncation},
\begin{displaymath}
  |S - \E S| \le |S_1 - \E S_1| + |S - S_1| + |\E S - \E S_1| \le |S_1 - \E S_1| + S_2 + \E S_2
\end{displaymath}
and so by \eqref{eq:unbounded-part-tail} and \eqref{eq:bounded-part-tail} applied to $t/2$ instead of $t$ we get
\begin{multline*}
  \p(|S - \E S| \ge t) \le \p(|S_1 - \E S_1| \ge t/2) + \p(S_2 + \E S_2 \ge t/2)\\
  \le C \exp\Big(-\frac{t}{2C U}\log\Big(1+\frac{tU}{2\Sigma^2}\Big)\Big) + \frac{2}{\Psi\big(\frac{t}{2CU}\big)+1}.
\end{multline*}

To finish the proof it remains to adjust the constants.
\end{proof}

Let us now pass to the proof of Proposition \ref{prop:weak-variance-necessity}

\begin{proof}[Proofs of Proposition \ref{prop:weak-variance-necessity}] Both parts of the proposition will be proved using the same testing variables. Fix $u > \Psi^{-1}(1)$ and consider a triangular array $(Y^{N}_i)_{N\ge 1,1\le i\le N}$ of independent random variables, such that $\p(Y^N_i = 1) = 1 - \p(Y^N_i = 0) = 1/(\Psi(u)N)$. Denote also by $Z$ a Poisson random variable with parameter $1/\Psi(u)$. Below $C$ stands for constants depending only on $\Psi$, which may change between occurrences.

Let $X_{N,i} = Y_{N,i} - \E Y_{N,i}$. Then, by Poisson's theorem, $S_N = \sum_{i=1}^N X_{N,i}$ converges in the total variation distance to $Z - \E Z$. In particular,
\begin{displaymath}
  \lim_{N\to \infty} \p(S_N \ge s) = \p(Z \ge s + 1/\Psi(u)).
\end{displaymath}
Thus for $s,u$ such that $s > 1 > 1/\Psi(u)$ and $N$ large enough
\begin{displaymath}
  \p(S_N\ge s) \ge \p(Z \ge \lceil 2s\rceil ).
\end{displaymath}
For $s,u$ large enough the above inequality gives
\begin{align}\label{eq:Poisson-lower-bound}
  \p(S_N\ge s) \ge \frac{1}{\lceil 2s\rceil!  \Psi(u)^{\lceil 2s\rceil }}e^{-1/\Psi(u)} \ge e^{-Cs\ln s - Cs\ln \Psi(u) - 1/\Psi(u)} \ge e^{-C^2(s\ln (1+s) + s \psi(u))},
\end{align}
where $C$ is some large constant, depending only on $\Psi$.

To bound $\p(S_N\ge s)$ from above, let us calculate the parameters appearing in \eqref{eq:Bennett-real-valued} and \eqref{eq:Bernstei-real-valued}.
We have $N\E X_{N,1}^2 \le \frac{1}{\Psi(u)}$.

Moreover,
\begin{displaymath}
\E \Psi\Big( u \max_{i\le N} Y_{N,i}\Big) = \Psi(u)\p(\max_{i\le N} Y_{N,i} = 1) \le 1,
\end{displaymath}
so $\|\max_{i\le N} Y_{N,i} \|_{\Psi} \le 1/u$, which due to Lemma \ref{le:mean-Orlicz} implies that $\|\max_{i\le N} |X_{N,i}|\|_{\Psi} \le C/u$.

Thus, an application of \eqref{eq:Bennett-real-valued} gives
\begin{align*}
\p(S_N\ge s) \le C\exp(-  c su \ln(1 + Cs\Psi(u)/u) + C\exp(-\psi(csu)),
\end{align*}
while \eqref{eq:Bernstei-real-valued} under the assumption that $\Psi$ is dominated by the exponential function leads to
\begin{align*}
\p(S_N\ge s) \le
C\exp( - c s^2\Psi(u)) + C\exp(-\psi(csu))
\end{align*}
(see Remark \ref{re:Bennett-Bernstein}, in particular \eqref{eq:Bernstein-equivalent}).

Comparing now the right hand-sides of the above two inequalities with the right-hand side of \eqref{eq:Poisson-lower-bound} and adjusting the constants, we obtain that under the assumptions of the first part of the proposition, for $s,u$ large enough and some positive constant $C$,
\begin{align}\label{eq:first-comparison}
C\Big(s\ln(1+s) + s\psi(u)\Big) \ge \min\Big(su\ln(1+s\Psi(u)/u), \psi(csu)\Big),
\end{align}
whereas in the setting of the second part, we get
\begin{align}\label{eq:second-comparison}
C\Big(s\ln(1+s) + s\psi(u)\Big) \ge \min\Big(s^2\Psi(u), \psi(csu)\Big)
\end{align}

Note that for $s,u$ sufficiently large, both $su\ln(1+s\Psi(u)/u)$ and $s^2\Psi(u)$ are greater than $C\Big(s\ln(1+s) + s\psi(u)\Big)$, thus both \eqref{eq:first-comparison} and \eqref{eq:second-comparison} imply that
\begin{displaymath}
  \psi(csu) \le C\Big(s\ln(1+s) + s\psi(u)\Big).
\end{displaymath}
To get \eqref{eq:nsc} it is now enough to adjust the constants.
\end{proof}

\subsection{Proofs of results from Section \ref{sec:convex-functions}\label{sec:proofs-convex-functions}}
The argument leading to Proposition \ref{prop:convex} will be similar to the one used for Theorem \ref{thm:weak-variance}, i.e. it will rely on splitting the variables into the bounded and unbounded part and using the Hoffmann-Jørgensen inequality to handle the unbounded one. In a similar form for specific Orlicz functions it appeared in \cite{MR2391154,MR4073683,10.1007/978-3-031-26979-0_7}.

\begin{proof}[Proof of Proposition \ref{prop:convex}]

In the proof, by $c,C$ we shall denote constants depending only on $\Phi$, which may vary from line to line. We will use the notation $|\cdot|$ to denote both the absolute value and the standard Euclidean norm on $\R^N$.

Let
\begin{displaymath}
M = \E \max_{i \le N}|X_i| \le  C\|\max_{i \le N}|X_i|\|_{\Phi},
\end{displaymath}
where the inequality follows from Lemma \ref{le:embedding}.

Define for $i\le N$, $Y_i = X_i 1_{\{|X_i| < 8M\}}$ and $Z_i = X_i - Y_i$. Set also $Y = (Y_i)_{i\le N}$, $Z = (Z_i)_{i\le N}$.

By the triangle inequality,
\begin{align}\label{eq:triangle}
 \p  (|f(X) - \E f(X)| \ge 2t) \le \p  (|f(Y) - \E f(X)| \ge t) + \p  (|f(X) - f(Y)| \ge t).
\end{align}

We shall first estimate $\||f(X)-f(Y)|\|_{\Phi}$. This will help us bound the second term on the right-hand side above as well as compare $\E f(X)$ and $\E f(Y)$.

Since $f$ is 1-Lipschitz, we have $|f(Y) - f(X)| \le |Z|$. Let us define $S_n = \sum_{i=1}^n Z_i^2$. Observe that by Markov's inequality,
\begin{displaymath}
\p  (\max_{n \le N}|S_n| > 0) = \p  (\max_{i \le N}Z_i^2 > 0) = \p  ( \max_{i \le N}|X_i| > 8M) \le \frac{1}{8}.
\end{displaymath}
Hence, by \eqref{eq:HJ-quantile} with $p = 1$, $\E |Z|^2 = \E S_N \le 8 \E \max_{i \le N}Z_i^2 \le C \|\max_{i \le N}Z_i^2\|_{\Psi}$

By \eqref{eq:nsc} and Theorem \ref{thm:characterization}, we know that \eqref{eq:HJTal} holds. Thus,
\begin{displaymath}
\| |Z|^2\|_{\Psi} \le D \Big(  \E |Z|^2 + \| \max_{i \le N} Z_i^2\|_{\Psi}\Big) \le C
\| \max_{i\le N} Z_i^2\|_{\Psi} \le C\|\max_{i\le N} |X_i|^2\|_{\Psi} = C\|\max_{i\le N} |X_i|\|_{\Phi}^2
\end{displaymath}
and so
\begin{align}\label{eq:Phi-norm-of-Z}
\||f(X) - f(Y)|\|_{\Phi} \le \||Z|\|_{\Phi} = \||Z|^2\|_{\Psi}^{1/2} \le C\|\max_{i\le N} |X_i|\|_{\Phi}.
\end{align}

In particular, for $t > 0$,
\begin{align}\label{eq:tail-of-Z}
\p  (|f(X) - f(Y)| \ge t) = \frac{2}{\Phi\Big(c \frac{t}{\|\max_{i\le N}|X_i|\|_{\Phi}}\Big) + 1}.
\end{align}

We shall now bound the first term on the right-hand side of \eqref{eq:triangle}. Let us start by estimating
\begin{displaymath}
  |\E f(X) - \E f(Y)| \le \E |f(X) - f(Y)| \le C\E \max_{i\le N}|Z_i| \le C\E \max_{i\le N} |X_i|,
\end{displaymath}
where in the second inequality we again used \eqref{eq:HJ-quantile} but this time in the space $\ell_2^n$ for the random variables $X_ie_i$, where $e_1,\ldots,e_n$ is the standard basis.

Applying Theorem \ref{thm:convex-Ledoux}, we thus get
\begin{multline*}
\p(|f(Y) - \E f(X)| \ge C\E \max_{i\le N} |X_i| + s) \\
\le \p(|f(Y) - \E f(Y)|\ge s) \le 2\exp\Big(-c\frac{s^2}{(\E \max_{i\le N} |X_i|)^2}\Big).
\end{multline*}

For $t > 2C\E \max_{i\le N} |X_i|$, applying this with $s = t - C\E \max_{i\le N} |X_i| \ge t/2$ yields
\begin{displaymath}
  \p(|f(Y) - \E f(X)| \ge t) \le 2\exp\Big(-c\frac{t^2}{(\E \max_{i\le N} |X_i|)^2}\Big)
\end{displaymath}
and together with \eqref{eq:tail-of-Z} and \eqref{eq:triangle} gives \eqref{eq:convex-unbounded}. On the other hand  for $t<2C\E \max_{i\le N} |X_i|$ and sufficiently small $c_\Psi$, the right-hand side of \eqref{eq:convex-unbounded} exceeds 1, and therefore \eqref{eq:convex-unbounded} holds trivially.

This ends the proof of the proposition.
\end{proof}

\subsection{Proofs of results from Section \ref{sec:boundedness}\label{sec:proofs-boundedness}}

\begin{proof}[Proof of Theorem \ref{thm:boundedness}]
The implication $(iii) \implies (ii)$ is trivial. Moreover, since $L_{\Psi}$ embeds into $L_1$ continuously (see Lemma \ref{le:embedding}), if $(ii)$ holds, by the original Hoffmann-Jørgensen's results (Theorem \ref{thm:HJ} or \eqref{eq:HJ-quantile}), we obtain that $\|S_n\|$ is bounded in $L_1$. By the assumption \eqref{eq:nsc} and Theorem \ref{thm:characterization} we thus obtain
\begin{displaymath}
  \sup_n \|S_n\|_{\Psi} \le C\Big(\sup_{n} \|S_n\|_1 + \Big\|\sup_n \|X_n\|\Big\|_{\Psi}\Big) < \infty,
\end{displaymath}
i.e., $(ii)$ implies $(i)$. It remains  to prove that $(i)$ implies $(iii)$. To this end let us first assume that $X_i$'s are symmetric. For $N \le \infty$ denote $M_N = \sup_{n< N} \|S_n\|$. By L\'{e}vy's inequality, for $N < \infty$ and any $t>0$,
\begin{displaymath}
\p(M_N \ge t) \le 2 \p  (\|S_N\| \ge t).
\end{displaymath} In particular, for any $\lambda > \sup_n \|S_n\|_{\Psi}$,
\begin{multline*}
 \E \Psi\Big(\frac{M_N}{\lambda}\Big) = \int_0^\infty \p  (M_N \ge \lambda \Psi^{-1}(t))dt \\
 \le 2\int_0^\infty \p  (\|S_N\| \ge \lambda \Psi^{-1}(t))dt = 2 \E \Psi\Big(\frac{\|S_N\|}{\lambda}\Big) \le 2
\end{multline*}
and so, by Lebesgue's monotone convergence theorem, $\E \Psi(M_\infty/\lambda) \le 2$, which by convexity gives $\E \Psi(M_\infty/(2\lambda)) \le 1$, i.e.,
\begin{displaymath}
\|M_\infty\|_\Psi \le 2\sup_n \|S_n\|_{\Psi} < \infty.
\end{displaymath}.

So far our proof followed very closely the original argument by Hoffmann-Jørgensen, with \eqref{eq:HJ-quantile} in $L_p$ replaced by \eqref{eq:HJTal} and \eqref{eq:HJ-quantile} for $p = 1$. In the remaining part we prefer to provide a more quantitative argument, which does not require additional tools other than those already introduced, and whose simple modification justifies our claim made below the equation \eqref{eq:HJ-quantile} about the (up to constant) formal equivalence of this inequality with seemingly weaker \eqref{eq:HJ-mean}.

Let us now consider the general case of not necessarily symmetric variables $X_n$. Let $(\varepsilon_n)_{n\ge 1}$ be a Rademacher sequence independent of  $(X_n)_{n\ge 1}$ and denote $\widetilde{S}_n = \sum_{i=1}^n \varepsilon_i (X_i - \E X_i)$. If $(S_n)_{n\ge 1}$ is bounded in $L_\Psi$, then by Lemmas \ref{le:symmetrization}, \ref{le:embedding} and \ref{le:mean-Orlicz}, so is $(\widetilde{S}_n)_{n\ge 1}$,
and since $\varepsilon_n X_n$'s are symmetric, we obtain
\begin{displaymath}
  \sup_{n\ge 1} \|\widetilde{S}_n\| \in L_{\Psi}.
\end{displaymath}
For a positive integer $N$, consider now the space $\ell_\infty^N(F) = F^N$, equipped with the maximum norm $\|x\|_\infty = \max_{i\le N} \|x_i\|$, together with independent random variables  $Z_i = (Z_{i1},\ldots,Z_{iN})$, where $Z_{ij} = 0$ for $j < i$ and $Z_{ij} = X_i - \E X_i$ for $j \ge i$. Note that
\begin{displaymath}
\Big\| \sum_{i=1}^N \varepsilon_i Z_i\Big\|_{\infty} = \sup_{n\le N} \|\widetilde{S}_n\|,\; \Big\|\sum_{i=1}^N Z_i\Big\|_{\infty} = \sup_{n\le N} \|S_n - \E S_n\|.
\end{displaymath}
Therefore, another application of Lemma \ref{le:symmetrization}, this time in $\ell_\infty^N(F)$ gives
\begin{displaymath}
  \Big\| \sup_{n\le N} \|S_n - \E S_n\|\Big\|_\Psi \le 2 \Big\|\sup_{n\le N} \|\widetilde{S}_n\|\Big\|_\Psi.
\end{displaymath}
Using the triangle inequality and Lemma \ref{le:mean-Orlicz}, we obtain that
\begin{displaymath}
  \Big\| \sup_{n\le N} \|S_n\|\Big\|_\Psi < C
\end{displaymath}
for some constant $C$, independent of $N$, which by Lebesgue's monotone convergence theorem  allows to conclude $(iii)$ and finish the proof.
\end{proof}

\begin{proof}[Proof of Proposition \ref{prop:boundedness-necessity}] For any $k \in \mathbb N_+$, consider numbers $N_k \in \mathbb N_+, u_k > 0$ which will be specified later on in such a way that $N_k\Psi(u_k) > 1$.

Similarly as in the proof of the first part of Theorem \ref{thm:characterization}, let us define $Y_{i,k}, k \in \mathbb N_+, i \le N_k$ as independent random variables with distribution
\begin{displaymath}
\p (Y_{i,k} = \pm u) = \frac{1}{2N_k\Psi(u_k)}, \;  \p(Y_{i,k}=0)=1 - \frac{1}{N_k\Psi(u_k)}
\end{displaymath}

From the elementary calculations in the proof of the first part of Theorem \ref{thm:characterization} we know that
\begin{align}\label{eq:bound-on-maximum}
\p  (\max_{i \le N_k}|Y_{i,k}| \neq 0) \le \frac{1}{\Psi(u_k)} \textrm{  and } \|\max_{i \le N_k} |Y_{i,k}|\|_{\Psi} \le 1.
\end{align}

Moreover by the reasoning therein we know that if \eqref{eq:nsc} does not hold, we can find sequences $u_k, N_k$, tending to $\infty$ (in particular we may assume that $N_k\Psi(u_k) > 1$), such that
\begin{displaymath}
  \Big\|\sum_{i=1}^{N_k} Y_{i,k}\Big \|_{\Psi} \ge k^3.
\end{displaymath}
We may moreover assume, by passing to a subsequence if necessary, that
\begin{align}\label{eq:convergent-series}
\sum_{k=1}^\infty \frac{1}{\Psi(u_k)} < \infty.
\end{align}

Now define
\begin{displaymath}
Z_{i,k} = \frac{Y_{i,k}}{k^2}, \;  m_k = \sum_{i=1}^{k-1} N_i \textrm{ for } k=1,2,\ldots,
\end{displaymath}
and then
\begin{displaymath}
X_j = Z_{j - m_{k},k}
\end{displaymath}
for $j \in \{m_{k}+1,...,m_{k+1}\}$, $k\ge 1$. It remains to check that $(X_j)_j$ satisfies all the conditions of Proposition \ref{prop:boundedness-necessity}

Firstly, $\p( X_j \neq 0 \textrm{ for infinitely many }j) = \p ( \max_{i \le N_k} Z_{i,k} \neq 0 \textrm{ for infinitely many } k)$. Moreover,
\begin{displaymath}
\sum_{k=1}^\infty \p  (\max_{i \le N_k} Z_{i,k} \neq 0) = \sum_{k=1}^\infty \p  (\max_{i \le N_k} |Y_{i,k}| \neq 0) < \infty
\end{displaymath}
by \eqref{eq:bound-on-maximum} and \eqref{eq:convergent-series}. Hence, by the Borel--Cantelli Lemma, with probability one, only finitely many $X_j$'s are nonzero, and so $(S_n)_n$ converges almost surely.

Further, by the triangle inequality and \eqref{eq:bound-on-maximum},
\begin{displaymath}
\Big\| \sup_n |X_n| \Big\|_{\Psi} \le \sum_{k=1}^\infty \frac{1}{k^2}\Big\|\max_{i \le N_k} |Y_{i,k}| \Big\|_{\Psi} \le \sum_{k=1}^\infty \frac{1}{k^2} < \infty
\end{displaymath} and so $\sup_n |X_n| \in L_\Psi(\mathbb R)$. It remains to note that by Jensen's inequality,
\begin{displaymath}
\Big\|\sum_{j=1}^{m_{k+1}} X_i \Big\|_{\Psi} \ge \Big\| \sum_{j = m_k+1}^{m_{k+1}} X_j \Big\|_{\Psi} = \frac{1}{k^2}\Big\| \sum_{i=1}^{N_k} Y_{i,k} \Big\|_{\Psi} \ge \frac{k^3}{k^2} = k,
\end{displaymath}
which means that $(S_n)_n$ is not bounded in $L_\Psi$.
\end{proof}

\bibliographystyle{amsplain}
\bibliography{HJTal}

\providecommand{\bysame}{\leavevmode\hbox to3em{\hrulefill}\thinspace}
\providecommand{\MR}{\relax\ifhmode\unskip\space\fi MR }
\providecommand{\MRhref}[2]{%
  \href{http://www.ams.org/mathscinet-getitem?mr=#1}{#2}
}
\providecommand{\href}[2]{#2}
\begin{thebibliography}{10}

\bibitem{MR2424985}
Rados{\l}aw Adamczak, \emph{A tail inequality for suprema of unbounded
  empirical processes with applications to {M}arkov chains}, Electron. J.
  Probab. \textbf{13} (2008), no. 34, 1000--1034. \MR{2424985}

\bibitem{MR2591905}
\bysame, \emph{A few remarks on the operator norm of random {T}oeplitz
  matrices}, J. Theoret. Probab. \textbf{23} (2010), no.~1, 85--108.
  \MR{2591905}

\bibitem{MR3423302}
Rados{\l}aw Adamczak and Witold Bednorz, \emph{Exponential concentration
  inequalities for additive functionals of {M}arkov chains}, ESAIM Probab.
  Stat. \textbf{19} (2015), 440--481. \MR{3423302}

\bibitem{MR0576407}
Aloisio Araujo and Evarist Gin\'{e}, \emph{The central limit theorem for real
  and {B}anach valued random variables}, John Wiley \& Sons, New
  York-Chichester-Brisbane, 1980. \MR{576407}

\bibitem{MR2123200}
St\'{e}phane Boucheron, Olivier Bousquet, G\'{a}bor Lugosi, and Pascal Massart,
  \emph{Moment inequalities for functions of independent random variables},
  Ann. Probab. \textbf{33} (2005), no.~2, 514--560. \MR{2123200}

\bibitem{MR3185193}
St\'{e}phane Boucheron, G\'{a}bor Lugosi, and Pascal Massart,
  \emph{Concentration inequalities}, Oxford University Press, Oxford, 2013, A
  nonasymptotic theory of independence, With a foreword by Michel Ledoux.
  \MR{3185193}

\bibitem{MR1890640}
Olivier Bousquet, \emph{A {B}ennett concentration inequality and its
  application to suprema of empirical processes}, C. R. Math. Acad. Sci. Paris
  \textbf{334} (2002), no.~6, 495--500. \MR{1890640}

\bibitem{chamakh:hal-03175697}
Linda Chamakh, Emmanuel Gobet, and Wenjun Liu, \emph{{Orlicz norms and
  concentration inequalities for $\beta$-heavy tailed random variables}},
  hal-03175697, June 2021.

\bibitem{MR4138129}
Linda Chamakh, Emmanuel Gobet, and Zolt\'{a}n Szab\'{o}, \emph{Orlicz random
  {F}ourier features}, J. Mach. Learn. Res. \textbf{21} (2020), Paper No. 145,
  37. \MR{4138129}

\bibitem{MR2434306}
Uwe Einmahl and Deli Li, \emph{Characterization of {LIL} behavior in {B}anach
  space}, Trans. Amer. Math. Soc. \textbf{360} (2008), no.~12, 6677--6693.
  \MR{2434306}

\bibitem{MR1857312}
Evarist Gin\'{e}, Rafa{\l} Lata{\l}a, and Joel Zinn, \emph{Exponential and
  moment inequalities for {$U$}-statistics}, High dimensional probability, {II}
  ({S}eattle, {WA}, 1999), Progr. Probab., vol.~47, Birkh\"{a}user Boston,
  Boston, MA, 2000, pp.~13--38. \MR{1857312}

\bibitem{MR2749436}
Nathael Gozlan, Cyril Roberto, and Paul-Marie Samson, \emph{From concentration
  to logarithmic {S}obolev and {P}oincar\'{e} inequalities}, J. Funct. Anal.
  \textbf{260} (2011), no.~5, 1491--1522. \MR{2749436}

\bibitem{MR3825894}
Nathael Gozlan, Cyril Roberto, Paul-Marie Samson, Yan Shu, and Prasad Tetali,
  \emph{Characterization of a class of weak transport-entropy inequalities on
  the line}, Ann. Inst. Henri Poincar\'{e} Probab. Stat. \textbf{54} (2018),
  no.~3, 1667--1693. \MR{3825894}

\bibitem{MR3706606}
Nathael Gozlan, Cyril Roberto, Paul-Marie Samson, and Prasad Tetali,
  \emph{Kantorovich duality for general transport costs and applications}, J.
  Funct. Anal. \textbf{273} (2017), no.~11, 3327--3405. \MR{3706606}

\bibitem{MR0356155}
J{\o}rgen Hoffmann-J{\o}rgensen, \emph{Sums of independent {B}anach space
  valued random variables}, Studia Math. \textbf{52} (1974), 159--186.
  \MR{356155}

\bibitem{MR4583676}
Han Huang and Konstantin Tikhomirov, \emph{On dimension-dependent concentration
  for convex {L}ipschitz functions in product spaces}, Electron. J. Probab.
  \textbf{28} (2023), Paper No. 63, 23. \MR{4583676}

\bibitem{MR0770640}
W.~B. Johnson, G.~Schechtman, and J.~Zinn, \emph{Best constants in moment
  inequalities for linear combinations of independent and exchangeable random
  variables}, Ann. Probab. \textbf{13} (1985), no.~1, 234--253. \MR{770640}

\bibitem{MR2135312}
T.~Klein and E.~Rio, \emph{Concentration around the mean for maxima of
  empirical processes}, Ann. Probab. \textbf{33} (2005), no.~3, 1060--1077.
  \MR{2135312}

\bibitem{MR4073683}
Yegor Klochkov and Nikita Zhivotovskiy, \emph{Uniform {H}anson--{W}right type
  concentration inequalities for unbounded entries via the entropy method},
  Electron. J. Probab. \textbf{25} (2020), Paper No. 22, 30. \MR{4073683}

\bibitem{MR0126722}
M.~A. Krasnoselskii and Ja.~B. Rutickii, \emph{Convex functions and {O}rlicz
  spaces}, russian ed., P. Noordhoff Ltd., Groningen, 1961. \MR{126722}

\bibitem{MR3263097}
Johannes Lederer and Sara van~de Geer, \emph{New concentration inequalities for
  suprema of empirical processes}, Bernoulli \textbf{20} (2014), no.~4,
  2020--2038. \MR{3263097}

\bibitem{MR1399224}
Michel Ledoux, \emph{On {T}alagrand's deviation inequalities for product
  measures}, ESAIM Probab. Statist. \textbf{1} (1995/97), 63--87. \MR{1399224}

\bibitem{MR1849347}
\bysame, \emph{The concentration of measure phenomenon}, Mathematical Surveys
  and Monographs, vol.~89, American Mathematical Society, Providence, RI, 2001.
  \MR{1849347}

\bibitem{MR1102015}
Michel Ledoux and Michel Talagrand, \emph{Probability in {B}anach spaces},
  Ergebnisse der Mathematik und ihrer Grenzgebiete (3) [Results in Mathematics
  and Related Areas (3)], vol.~23, Springer-Verlag, Berlin, 1991, Isoperimetry
  and processes. \MR{1102015}

\bibitem{MR1782276}
Pascal Massart, \emph{About the constants in {T}alagrand's concentration
  inequalities for empirical processes}, Ann. Probab. \textbf{28} (2000),
  no.~2, 863--884. \MR{1782276}

\bibitem{MR2391154}
Shahar Mendelson and Nicole Tomczak-Jaegermann, \emph{A subgaussian embedding
  theorem}, Israel J. Math. \textbf{164} (2008), 349--364. \MR{2391154}

\bibitem{MR1113700}
M.~M. Rao and Z.~D. Ren, \emph{Theory of {O}rlicz spaces}, Monographs and
  Textbooks in Pure and Applied Mathematics, vol. 146, Marcel Dekker, Inc., New
  York, 1991. \MR{1113700}

\bibitem{10.1007/978-3-031-26979-0_7}
Holger Sambale, \emph{Some notes on concentration for $\alpha$-subexponential
  random variables}, High Dimensional Probability IX (Cham) (Rados{\l}aw
  Adamczak, Nathael Gozlan, Karim Lounici, and Mokshay Madiman, eds.), Springer
  International Publishing, 2023, pp.~167--192.

\bibitem{MR1756011}
Paul-Marie Samson, \emph{Concentration of measure inequalities for {M}arkov
  chains and {$\Phi$}-mixing processes}, Ann. Probab. \textbf{28} (2000),
  no.~1, 416--461. \MR{1756011}

\bibitem{MR0430667}
S.~J. Szarek, \emph{On the best constants in the {K}hinchin inequality}, Studia
  Math. \textbf{58} (1976), no.~2, 197--208. \MR{430667}

\bibitem{MR1048946}
Michel Talagrand, \emph{Isoperimetry and integrability of the sum of
  independent {B}anach-space valued random variables}, Ann. Probab. \textbf{17}
  (1989), no.~4, 1546--1570. \MR{1048946}

\bibitem{MR1361756}
\bysame, \emph{Concentration of measure and isoperimetric inequalities in
  product spaces}, Inst. Hautes \'{E}tudes Sci. Publ. Math. (1995), no.~81,
  73--205. \MR{1361756}

\bibitem{MR1419006}
\bysame, \emph{New concentration inequalities in product spaces}, Invent. Math.
  \textbf{126} (1996), no.~3, 505--563. \MR{1419006}

\bibitem{MR1387624}
\bysame, \emph{A new look at independence}, Ann. Probab. \textbf{24} (1996),
  no.~1, 1--34. \MR{1387624}

\bibitem{MR3101846}
Sara van~de Geer and Johannes Lederer, \emph{The {B}ernstein-{O}rlicz norm and
  deviation inequalities}, Probab. Theory Related Fields \textbf{157} (2013),
  no.~1-2, 225--250. \MR{3101846}

\bibitem{MR1385671}
Aad~W. van~der Vaart and Jon~A. Wellner, \emph{Weak convergence and empirical
  processes}, Springer Series in Statistics, Springer-Verlag, New York, 1996,
  With applications to statistics. \MR{1385671}

\bibitem{MR3707425}
Jon~A. Wellner, \emph{The {B}ennett--{O}rlicz norm}, Sankhya A \textbf{79}
  (2017), no.~2, 355--383. \MR{3707425}

\bibitem{MR1464686}
Klaus Ziegler, \emph{On {H}offmann-{J}{\o}rgensen-type inequalities for outer
  expectations with applications}, Results Math. \textbf{32} (1997), no.~1-2,
  179--192. \MR{1464686}

\end{thebibliography}
\end{document}